\documentclass[11pt]{article}
\usepackage{hyperref}
\usepackage[T1]{fontenc}
\usepackage{amsfonts}
\usepackage[english]{babel}
\usepackage{a4wide,times}

\usepackage[latin1]{inputenc}
\usepackage{amssymb}
\usepackage{epsfig}
\usepackage{amsbsy}
\usepackage{verbatim}
\usepackage{color}
\usepackage{mathrsfs}
\usepackage{graphicx}
\usepackage{epstopdf}
\usepackage{makeidx}
\usepackage{tikz}
\usepackage{amsmath}
\usepackage{amsthm}
\newcommand{\prods}[2]{\left(#1\middle|#2\right)}

\theoremstyle{plain}
\newtheorem{thm}{Theorem}[section]

\newtheorem{prop}[thm]{Proposition}

\newcommand{\argmin}{\arg\!\min}

\theoremstyle{definition}
\newtheorem{defn}{Definition}[section]
\usepackage{dsfont}

\theoremstyle{remark}
\newtheorem{rem}{\bf Remark}[section]
\theoremstyle{remark}

\newtheorem{com*}{\bf Comment}

\usepackage{fancyhdr}

\makeatletter
\def \newequation#1#2{
   \@definecounter{#1}
   \@namedef{the#1}{\hbox{#2}}
   \@namedef{#1}{$$\refstepcounter{#1}}
   \@namedef{end#1}{
      \eqno \csname the#1\endcsname $$\global\@ignoretrue
      }
}
\makeatother
\newequation{E1}{($E_{b,\sigma,W}$)}
   
\makeatletter
\def \newequation#1#2{
   \@definecounter{#1}
   \@namedef{the#1}{\hbox{#2}}
   \@namedef{#1}{$$\refstepcounter{#1}}
   \@namedef{end#1}{
      \eqno \csname the#1\endcsname $$\global\@ignoretrue
      }
   }
\makeatother
\newequation{hyp3}{($\mathcal{H}_{b,\sigma}$)}

\makeatletter
\def \newequation#1#2{
   \@definecounter{#1}
   \@namedef{the#1}{\hbox{#2}}
   \@namedef{#1}{$$\refstepcounter{#1}}
   \@namedef{end#1}{
      \eqno \csname the#1\endcsname $$\global\@ignoretrue
      }
   }
\makeatother
\newequation{shleg}{(ShLeg)}

\makeatletter
\def \newequation#1#2{
   \@definecounter{#1}
   \@namedef{the#1}{\hbox{#2}}
   \@namedef{#1}{$$\refstepcounter{#1}}
   \@namedef{end#1}{
      \eqno \csname the#1\endcsname $$\global\@ignoretrue
      }
   }
\makeatother
\newequation{kl}{(KL)}

\makeatletter
\def \newequation#1#2{
   \@definecounter{#1}
   \@namedef{the#1}{\hbox{#2}}
   \@namedef{#1}{$$\refstepcounter{#1}}
   \@namedef{end#1}{
      \eqno \csname the#1\endcsname $$\global\@ignoretrue
      }
   }
\makeatother
\newequation{haar}{(Haar)}

\title{Product Markovian quantization  of  an $\mathbb R^d$-valued   Euler  scheme of a diffusion process  with applications to finance}

\author{ 
{\sc Lucio Fiorin} \thanks{Department of Mathematics,
University of Padova, via Trieste 63, 35121 Padova, Italy. Email: {\tt fiorin@math.unipd.it}.} \ \ \
{\sc  Gilles Pag\`es} \thanks{Laboratoire de Probabilit\'es et Mod\`eles Al\'eatoires (LPMA), UPMC-Sorbonne Universit\'e, UMR  CNRS 7599, case 188, 4, pl. Jussieu, F-75252 Paris Cedex 5, France.
E-mail: {\tt gilles.pages@upmc.fr} } \ \ \ 
{\sc  Abass Sagna} \thanks{ENSIIE \& Laboratoire de Math\'ematiques et Mod\'elisation d'Evry (LaMME), Universit\'e  d'Evry Val-d'Essonne,  UMR CNRS  8071,    23 Boulevard de France, 91037 Evry. E-mail: {\tt abass.sagna@ensiie.fr}. }\ \ \
\thanks{The first author benefited from the support of the Fondazione Aldo Gini of the University of Padova. The second  author  benefited from the support of the Chaire ``Risques financiers'',  a joint initiative of \'Ecole Polytechnique, ENPC-ParisTech and UPMC, under the aegis of the Fondation du Risque. The third author  benefited from the support of the Chaire ``Markets in Transition'', under the aegis of Louis Bachelier Laboratory, a joint initiative of \'Ecole polytechnique, Universit\'e d'\'Evry Val d'Essonne and  F\'ed\'eration Bancaire Fran\c{c}aise.} 
}

\date{}

\begin{document}

\maketitle

\begin{abstract}
We introduce a new approach to quantize the Euler scheme of an $\mathbb R^d$-valued  diffusion  process. This method is based on a Markovian and  componentwise product quantization and allows us, from a numerical point of view, to speak of {\em fast online quantization} in dimension greater than one since the product quantization of the Euler scheme of the diffusion  process and its companion weights and transition probabilities  may be computed quite instantaneously. We show that the resulting quantization process is a Markov chain, then, we  compute the associated  companion weights and transition probabilities  from (semi-) closed formulas.  From the analytical point of view, we show that the induced quantization errors at  the  $k$-th discretization step $t_k$ is a cumulative of the marginal quantization error up  to time $t_k$.    Numerical experiments are performed  for the pricing of a Basket call option, for the pricing of  a European call option in a Heston model  and for the approximation of the solution of  backward stochastic differential equations  to show  the performances of the method.  
\end{abstract}


\section{Introduction}
In  \cite{PagSagMQ} is  proposed and analyzed a Markovian  (fast) quantization of an $\mathbb R^d$-valued Euler scheme of a diffusion  process. However,  in practice, this approach allows to speak of fast quantization only in dimension one since, as soon as $d \ge 2$,   one has to use   recursive stochastic zero search  algorithm (known to be very time consuming, compared to deterministic procedures like the Newton-Raphson algorithm, see \cite{PagPri03}) to compute optimal quantizers, their associated weights and transition  probabilities.  In order to overcome this limitation, we propose  in this work another approach to quantize an $\mathbb R^d$-valued Euler scheme of a diffusion  process.  This method is based on a Markovian and  componentwise product quantization.  It  allows again to speak of fast quantization in high dimension since the product quantization of the Euler  scheme of the diffusion   process and its   transition probabilities  can be computed almost instantaneously still  using deterministic  recursive zero search algorithms.

In a general setting, the stochastic process  $(X_t)_{t\in [0,T]}$ of interest   is  defined as  the  (strong) solution to  the following stochastic differential equation  
\begin{equation} \label{EqSignalIntro}
 X_t = X_0+\int_0^t b(s,X_s) ds + \int_0^t\sigma(s,X_s) dW_s
 \end{equation}
where  $W$ is a  standard  $q$-dimensional  Brownian motion,   independent from the $\mathbb R^d$-valued random vector $X_0$, both defined on the same probability space $(\Omega,{\cal A}, \mathbb P)$. The drift  coefficient  $b:[0,T] \times \mathbb R^d \rightarrow \mathbb R^d$ and the volatility  coefficient   $\sigma:[0,T] \times \mathbb R^d \rightarrow \mathbb R^{d \times q}$ are Borel measurable functions  satisfying  appropriate Lipschitz continuity  and linear growth conditions (specified further on) which  ensure  the existence of  a unique strong solution of the stochastic differential equation. In corporate  finance, these processes are used to model the dynamics  of assets for several quantities of interest involving   the pricing and the  hedging  of derivatives.  These quantities are usually  of the form  
  \begin{equation}  \label{EqPriceEuroLikeOptionsC}
    \mathbb E \big[f( X_{T}) \big], \quad T>0, 
    \end{equation}
 or 
\begin{equation}  \label{EqNLfilterLikeOptionsC}
 \mathbb E \big[f(X_t) \vert   X_{s} = x \big] , \quad  0<s <t,
 \end{equation}
for a given Borel function $f: \mathbb R^d \rightarrow \mathbb R$.  For illustrative purposes, let us consider  the following three   examples which may be reduced to the computation of regular expectations like   \eqref{EqPriceEuroLikeOptionsC} or \eqref{EqNLfilterLikeOptionsC}.  First, let us consider the price of   a Basket call option with  maturity $T$ and  strike $K$, based on two stocks whose  prices $S^1$ and $S^2$   evolve following the dynamics
\begin{equation}   \label{EqBasketOptionIntro}
   \left \{ \begin{array}{l}

 dS_t^1  = r S_t^1  +  \rho \, \sigma_1  S_t^1 dW_{t}^1 +   \sqrt{1-\rho^2}\, \sigma_1 S_t^1  dW_t^2  \\
    dS_t^2  = r S_t^2 dt  + \sigma_2 S_t^2 dW_{t}^1  
    \end{array}  
 \right.
 \end{equation}
 where $r$ is the interest rate, $\sigma_1, \sigma_2>0$,  $\rho \in (-1,1)$ is a correlation term and  $W^1$ and $W^2$ are two independent Brownian motions.  We know that the no arbitrage price  at time $t=0$ in a complete market reads
\begin{equation}   \label{EqpriceBasketIntro}
e^{-rT} \mathbb E \big[ ( w_1 S^1_T + w_2 S^2_T -K)_{+} \big] = e^{-rT}  \mathbb E F(X_T), \quad  X = (S^1,S^2), 
\end{equation}
where  the weights $w_1$ and $w_2$ are usually assumed to be positive with a sum  equal to one  and where the function $F$ is defined, for every $x = (s^1,s^2) \in \mathbb R^2$, by  $F(x)  = (w_1s^1 + w_2 s^2 -K)_{+}$. Keep in mind that $x_{+} = \max(x, 0)$, $x \in \mathbb R$.

The second example concerns the  pricing of a call option with maturity $T$ and strike $K$, in the Heston model of \cite{Heston93} where  the stock price $S$  and its  stochastic variance $V$ evolve following the  (correlated) dynamics
\begin{equation}  \label{EqHestonIntro}
   \left \{ \begin{array}{l}
   dS_t  = r  S_t dt +   \rho \, \sqrt{V_t} S_t dW_{t}^1 + \sqrt{1-\rho^2}\, \sqrt{V_t}  S_{t} dW_t^2\\
 dV_t  = \kappa(\theta -V_t) dt  +   \sigma  \sqrt{V_t} dW_{t}^1 , \quad t \in [0,T].  
    \end{array}  
 \right.
 \end{equation}
In the previous equation, the parameter $r$ is still  the interest rate; $\kappa>0$ is the rate at which  $V$ reverts to the long  run average  variance $\theta>0$;  the parameter $\sigma>0$ is the volatility of the variance and $\rho \in [-1,1]$ is  the correlation term.  In this case,  the no arbitrage price   at time $t=0$ in a complete market  reads  under this risk neutral probability 
\begin{equation} \label{EqpriceHestonIntro}
e^{-rT} \mathbb E \big[(S_T - K)_{+} \big] = e^{-rT} \mathbb E H(X_T), \quad X=(S,V),
\end{equation}
where $H(x) = (x^1-K)_{+}$, for $x = (x^1,x^2) \in \mathbb R^2$. 

Note that both price expressions \eqref{EqpriceBasketIntro} and \eqref{EqpriceHestonIntro} are  of type  \eqref{EqPriceEuroLikeOptionsC}.

The last example concerns the approximation of the following Backward Stochastic Differential Equation (BSDE),  
   \begin{equation}   \label{EqBSDEIntro11}
 Y_t  =  \xi  + \int_t^T f(s,X_s,Y_s,Z_s) ds - \int_t^T Z_s \cdot dW_s,\quad t \in [0,T], 
 \end{equation}
where terminal condition of the form  $\xi = h(X_T)$ for  a given Borel function $h:\mathbb R^d \rightarrow \mathbb R$ and  where $X$ is a  strong solution to \eqref{EqSignalIntro}.  The process   $(Z_t)_{t \in [0,T]}$ is a square integrable progressively measurable process taking values in $\mathbb R^q$ and   $f : [0,T]{\small \times} \mathbb R^d $  ${\small \times} \mathbb R {\small \times \mathbb R^q} \rightarrow \mathbb R$ is a Borel function.  We will see further on that the approximation of the solution of the BSDE \eqref{EqBSDEIntro11} involves the computation of expressions of the type \eqref{EqNLfilterLikeOptionsC}.

In the general setting (in particular,  in both previous examples   \eqref{EqBasketOptionIntro}-\eqref{EqpriceHestonIntro})   the stochastic differential equation  (\ref{EqSignalIntro}) has no  explicit solution. Therefore,   both   quantities (\ref{EqPriceEuroLikeOptionsC}) and (\ref{EqNLfilterLikeOptionsC})  have to  be  approximated, for example,   by
  \begin{equation}  \label{EqPriceEuroLikeOptions}
    \mathbb E \big[f(\bar X_{T})\big] 
    \end{equation}
and
\begin{equation}  \label{EqNLfilterLikeOptions}
 \mathbb E \big[f(\bar X_{t_{k+1}}) \vert  \bar X_{t_k} = x \big] \quad \textrm{where  } t=t_{k+1},  s=t_k,
 \end{equation}
and where $(\bar  X_{t_k})_{k=0,\ldots,n}$ is a discretization scheme of the process $(X_t)_{t \ge 0}$ on  $[0,T]$, for a given discretization mesh $t_k = k  \Delta$, $k=0,\ldots, n$, $\Delta = T/n$. The Euler scheme   $(\bar  X_{t_k})_{k=0,\ldots,n}$ associated to  $(X_t)_{t\in [0,T]}$  is recursively  defined  by 
$$\bar X_{t_{k+1}}= \bar X_{t_k} + b(t_k,\bar X_{t_k}) \Delta  + \sigma(t_k,\bar X_{t_k}) (W_{t_{k+1}} - W_{t_k}), \quad  \bar X_0 = X_0.$$   
In the sequel, when no confusion may occur,     we will identify  the value $Y_{t_k}$ at time $t_k$ of any   process $(Y_{t_k})_{0 \le k \le n}$     by  $Y_{k}$, $k=0, \ldots,n$.

At this stage, the quantities \eqref{EqPriceEuroLikeOptions} and \eqref{EqNLfilterLikeOptions} still have  no closed formulas in the general setting  so is  the case when for example dealing with a general local volatility  model or  a stochastic volatility model like the Heston model. Consequently,   we have to make a spacial approximation of the expectation or the conditional expectation. This may be done by Monte Carlo simulation techniques or by optimal quantization method in particular, by the Markovian (fast) quantization method. 

The fast Markovian  quantization of the Euler scheme of an $\mathbb R^d$-valued  diffusion process  has been introduced in \cite{PagSagMQ}.  It consists of  a sequence of   quantizations $(\hat X_{k}^{\Gamma_k}  )_{k=0,\ldots,N}$ of the Euler scheme   $(\bar X_{k}  )_{k=0,\ldots,N}$  defined recursively as follows: 
 \begin{eqnarray*}
  \tilde X_0  &= &\bar X_0, \\
\hat X_{k}^{\Gamma_k}   &= &{\rm Proj}_{\Gamma_k}(\tilde X_{k})    \quad \textrm{and}\quad   \tilde X_{k+1} = {\cal E}_k(\hat X_{k}^{\Gamma_k} ,Z_{k+1}) ,\; k=0,\ldots,n-1,
  \end{eqnarray*}
where $(Z_k)_{k=1,\ldots, n}$ is an  i.i.d. sequence of  ${\cal N}(0;I_q)$-distributed random vectors, independent of  $\bar X_0$ and 
\[
   {\cal E}_{k}(x,z)  = x + \Delta b(t_{k},x) + \sqrt{\Delta} \sigma(t_{k},x) z, \quad x \in \mathbb R^d, \ z \in \mathbb R^q, \ k=0, \ldots, n-1.
   \]
  The sequence of quantizers satisfies for every $k \in \{0, \ldots,n\}$,
\[
 \Gamma_k \!\in  \argmin \{ \tilde D_{k}(\Gamma),\  \Gamma \subset \mathbb R^d, \ {\rm card}(\Gamma) \leq N_{k} \},
 \]
where for every grid $\Gamma \subset \mathbb R^d$, $\tilde D_{k+1} (\Gamma)  :=  \mathbb E \big[   {\rm dist}(\tilde X_{t_{k+1}}, \Gamma)^2 \big]$. However, this quantization method is  fast from the numerical point of view only in one  dimension. 

The aim of this work is to present  another approach to quantize the Euler scheme of an $\mathbb R^d$-valued diffusion process in order to speak of fast quantization in dimension greater than one.  We propose a Markovian and product quantization method. It allows us to compute instantaneously  the  optimal product quantizers   and their transition probabilities (and its companion  weights)  when the size of the quantizations are chosen reasonably.

The method is based on a Markovian and componentwise product quantization  of the process $(\bar X_k)_{0 \le k \le n}$. To be more precise,  let us denote by $\Gamma_k^{\ell}$ an  $N_k^{\ell}$-quantizer  of   the $\ell$-th component  $\bar X_k^{\ell}$ of the vector $\bar X_k$ and let  $\hat X_k^{i}$ be the  quantization  of $\bar X_k^{\ell}$ of size $N_k^{\ell}$,  on the grid $\Gamma_k^{\ell}$.    Let us   define the  product quantizer $\Gamma_k =  \bigotimes_{i=1}^d  \Gamma_k^{\ell}$ of size  $N_k = N_k^1  {\small \times }  \ldots {\small \times }  N_{k}^d$ of  the vector $\bar X_k$  as   
  \begin{eqnarray*}
  \Gamma_k &  = &  \big\{ (x_k^{1,i_1}, \ldots,x_k^{d,i_d}),   \quad  i_{\ell} \in \{1, \ldots,N_k^{\ell}\},  \  \ell \in \{1, \ldots,d\}    \big\}.  
  \end{eqnarray*}
Then,  assuming that  $\bar X_0$ is already  quantized as $\hat X_0$,  we define the product quantization $(\hat X_{t_k})_{0 \le k \le n}$ of the process $(\bar X_{t_k})_{0 \le k \le n}$  from the following recursion: 
   \begin{equation}  \label{EqAlgorithmIntro}
   \left \{ \begin{array}{l}
 \tilde X_0 = \hat X_0, \quad  \hat X_k ^{\ell} = {\rm Proj}_{\Gamma_k^{\ell}}(\tilde X_k^{i}), \ i=1, \ldots, d   \\
   \hat X_k  =  ( \hat X_k^1, \ldots, \hat X_k^d)  \quad  \textrm{and} \quad  \tilde X_{k+1}^{\ell}  =  {\cal E}_k^{\ell}(\hat X_k,Z_{k+1}), \ i=1, \ldots, d \\
   {\cal E}_k^{\ell} (x,z)= x^{\ell} + \Delta b^{\ell}(t_k,x) + \sqrt{\Delta} (\sigma^{\ell  \bullet}(t_k,x) \vert  z),  \ z = (z^1, \ldots,z^q)\in \mathbb R^q\\
     x=(x^1, \ldots,x^d), \ b=(b^1, \ldots, b^d)   \textrm{ and } (\sigma^{\ell  \bullet}(t_k,x) \vert  z) =\sum_{m=1}^q \sigma^{\ell m}(t_k,x) z^m
 \end{array}  
 \right.
 \end{equation}
 where for $a \in {\cal M}(d,q)$, $a^{\ell \bullet}  = [a_{\ell j}]_{j=1, \ldots,q}$.

 It is easy to see that the sequence of quantizers $(\hat X_{k})_{k \ge 0}$ is a Markov chain (see Remark \ref{PropMarkovChainProper}). Then, the challenging question is to know  how to compute its (set) values, $i.e.$ the product quantizer $\Gamma_k =  \bigotimes_{i=1}^d  \Gamma_k^{\ell}$, $k=0, \ldots n$, and  the associated   transition probabilities. Using the fact that the conditional distribution of the Euler scheme is a multivariate Gaussian distribution and that  each component of a Gaussian vector remains a scalar Gaussian random variable, we propose a way to quantize every component $\bar X_k^{\ell}$ of the vector $\bar X_k$, for $k=0, \ldots,n$.  We then define the product quantization $(\hat X_{k})_{0 \le k \le n}$ of $(\bar X_k)_{0 \le k \le n }$  from the recursive procedure \eqref{EqAlgorithmIntro}. Then,   we  show how to  compute, for every $k \ge 1$,  the  companion  transition probabilities (and the companion  weights)  associated to each component of the vector $\hat X_{k}$, for every $k \ge 1$ and to the vector $\hat X_{k}$ itself.  
 
When the components of  the vector $\bar X_{k}$ are independent for every $k=0, \ldots, n$, the method boils down to the usual product quantization of the vector $\bar X_k$,  where each component  is quantized from the Markovian recursive quantization method (see~\cite{PagSagMQ}). In this case, the transition probability weight associated to the vector $\hat X_k$ is the product  of the transition probability weights associated to its components.
 
  The main difficulties  arise when the components of $\bar X_k$ are not independent. In this work, we propose a closed formula even in this case by relying on  a domain decomposition  technique.
 
  To be more precise,  set, for every $k \in \{ 0, \ldots,n \}$,
  \begin{equation} \label{Eqmultindex}
  I_k = \big\{(i_1, \ldots,i_d), \ i_{\ell} \in \{1, \ldots, N_k^{\ell} \} \big\} 
  \end{equation}
  and for  $i := (i_1, \ldots,i_d)  \in I_k$,  set
  \begin{equation} \label{EqmultindexComp}
    x_k^{i} := (x_k^{1,i_1}, \ldots,x_k^{d,i_d}).
  \end{equation}

We will  show in Proposition \ref{PropTransitionProba}  that the transition probabilities of the Markov chain $(\hat X_{k})_{k \ge 0}$ reads, for any multi-indices $i \in I_k$ and $j  \in I_{k+1}$,
\begin{eqnarray} \label{EqProbaCondVectorIntro}
 \mathbb P\big(\widehat X_{k+1} = x_{k+1}^{j} \vert \widehat X_{k}  = x_k^{i} \big)  & = &  \mathbb E \Big[ \prod_{ \ell \in \mathbb J^{0}_k(x^{i}_k)} \mbox{\bf{1}}_{\{  \zeta \in \mathbb J^0_{k,j_{\ell}}(x^{i}_k)  \}} \max\big( \Phi_0(\beta_{j} (x_k^{i},\zeta))  -   \Phi_0(\alpha_{j}(x_k^{i},\zeta)),0\big) \Big]   \qquad
\end{eqnarray}
where $\zeta \sim {\cal N}(0; I_{q-1})$, with the convention that $\prod_{\ell \in \emptyset} (\cdot) = 1$. The function $\Phi_0$ stands for the cumulative distribution function of the ${\cal N} (0,1)$ and the functions $\alpha_j$ and $\beta_j$ are defined in Proposition \ref{PropTransitionProba}.  Although  this formula looks  quite involved, it turns out to be very simple to use in practice. In fact, keeping in mind that the optimal  quantization grids associated to  multivariate Gaussian random vectors (up to dimension $d=10$) can  be downloaded  on the  website {\tt www.quantize.maths-fi.com}, it is clear that \eqref{EqProbaCondVectorIntro} can be  computed  instantaneously using these optimal grids of multivariate normal vectors. Furthermore, Equation \eqref{EqProbaCondVectorIntro} allows us to deduce the weights associated to the product quantization $\hat X_{k+1}$, $k=0, \ldots,n-1$, since for every $j  \in I_{k+1}$,
\begin{equation} \label{EqWeightsIntro} 
\mathbb P \big(\hat X_{k+1}  = x_{k+1}^{j}  \big)   =  \sum_{i  \in I_k}   \mathbb P\big(\hat X_{k+1} = x_{k+1}^{j} \vert \hat X_{k}  = x_k^{i} \big)     \mathbb P \big(\hat X_{k} = x_k^{i}). 
\end{equation}

 Formulas \eqref{EqProbaCondVectorIntro}-\eqref{EqWeightsIntro} are  useful when,  for example, we deal with the approximation of the solution of BSDEs or when we deal with  the pricing  of a Basket call like  Equation \eqref{EqpriceBasketIntro}. In this last situation, given a time discretization mesh $t_0=0, \ldots, t_n=T$, the price of the Basket call option will be approximated by the cubature formula
 
 \begin{equation}  \label{EqApproxBasketCall}
 e^{-rT} \sum_{j  \in I_{n}} F(x_n^{j}) \, \mathbb P(\hat X_n = x_n^j),
 \end{equation}
 where the probability weights $\mathbb P(\hat X_n = x_n^j)$ are computed recursively  in a forward way using  equations \eqref{EqProbaCondVectorIntro} and \eqref{EqWeightsIntro}.  
 
  When the correlation coefficient $\rho = 0$ in \eqref{EqBasketOptionIntro}, the probabilities $\mathbb P(\hat X_n = x_n^j)$ in  \eqref{EqApproxBasketCall}   will be computed  in a simplified since  the components of the vectors  $\bar X_k$ are  independent.  We show in Proposition \ref{PropCasIndep} that the formula  \eqref{EqProbaCondVectorIntro} reads in the simplified form
\begin{eqnarray*} \label{EqProductProbaIntro}
 \mathbb P\big(\hat X_{k+1} = x_{k+1}^{j} \vert \hat X_{k}  = x_k^{i} \big) &=& \prod_{\ell=1}^d   \mathbb P\big(\hat X_{k+1}^{\ell} = x_{k+1}^{j_{\ell}} \vert \hat X_{k}  = x_k^{i} \big)  \\
  & = &  \prod_{\ell=1}^d  \big(\Phi_0 \big( x_{k+1}^{\ell,j_{\ell}+}(x_k^{i},0) \big)  -   \Phi_0 \big( x_{k+1}^{\ell,j_{\ell}-}(x_k^{i},0) \big) \big),
\end{eqnarray*}
for $i \in I_k, j \in I_{k+1}$.

 For the approximation of the solution of the BSDE \eqref{EqBSDEIntro11},  many numerical schemes and approximating methods have been proposed (see e.g.~\cite{BalPagPri0, BouTou, CriManTou, GobTur, HuNuaSon, GobLopTurVaz, BenDen}). In this paper, we just aim to test the numerical performances of our method using the (quantization) numerical scheme proposed in~\cite{PagSagBSDE}. Setting $\hat Y_k = \hat y_k(\hat X_k)$ (where $(\hat X_k)$ is the  quantization of the Euler scheme $(\bar X_{t_k})$), for every $k \in \{ 0, \cdots,n\}$, this quantized BSDE scheme reads  (see~\cite{PagSagBSDE}) as
\begin{equation*}
   \left \{ \begin{array}{l}
   \hat y_{n}(x_n^i)  =     h( x_{n}^i) \hspace{6.04cm }  x_n^i  \in\Gamma_n\\
   \hat y_{k}(x_k^i)   = \hat {\alpha}_k(x_k^i)   + \Delta_n f \big(t_k, x_{k}^i, \hat {\alpha}_k(x_k^i) , \hat {\beta}_k(x_k^i)  \big) \qquad  \   x_k^i \in \Gamma_k
    \end{array}  
 \right.
 \end{equation*}
 where for $\ k=0, \ldots,n-1$, for $i \in I_{k}$,
\begin{equation}
\hat{\alpha}_k(x_k^i) = \sum_{j \in I_{k+1}}  \hat y_{k+1}(x_{k+1}^j) \,  p_k^{ij} \quad \textrm{ and } \quad \hat {\beta}_k(x_k^i)   =  \frac{1}{\sqrt{\Delta_n}} \sum_{j\in I_{k+1}}    \hat y_{k+1}(x_{k+1}^j) \, \Lambda_k^{ij},
\end{equation}
with 
\[
p_k^{ij} = \mathbb P(\hat X_{k+1} = x_{k+1}^j \vert \hat X_{k} = x_k^i) \quad  \textrm{  and  }  \quad   \Lambda_{k}^{ij}    =  \mathbb{E}\big(Z_{k+1}\mathbb { 1}_{\{ \hat X_{k+1}=x_{k+1}^j  \}}\big \vert  \hat X_k=x_k^i \big).
\]
We will show further on how to compute the previous quantities from the Markovian product quantization method and using (semi)-closed formulas.

  We also compute  the  (transition) distribution of each component  of the product  quantizations.  Indeed,  we show in Proposition \ref{PropTransComponent} that for every $\ell \in \{1, \ldots,d\}$ and for every   $j_{\ell} \in \{1, \ldots,N_{k+1}^{\ell} \}$,  the transition probability $\mathbb P(\tilde X_{k+1}^{\ell} \in C_{j_{\ell}}(\Gamma_{k+1}^{\ell})\vert \hat X_{k}  =x_k^{i})$ is given by
\begin{eqnarray}  \label{EqProbaCondComp_iIntro}
 \mathbb P\big(\tilde X_{k+1}^{\ell} = x_{k+1}^{\ell j_{\ell}} \vert \hat X_{k}  = x_k^{i} \big) 
 & =&  \Phi_0 \Bigg(  \frac{x_{k+1}^{\ell,j_{\ell}+1/2} - m_k^{\ell}(x_k^{i})}{\sqrt{\Delta} \, \vert  \sigma_k^{\ell {\scriptscriptstyle \bullet}}(x_k^{i}) \vert_{_2}}  \Bigg) -  \Phi_0 \Bigg(  \frac{x_{k+1}^{\ell,j_{\ell}-1/2} - m_k^{\ell}(x_k^{i})}{\sqrt{\Delta} \,  \vert   \sigma_k^{\ell {\scriptscriptstyle \bullet}}(x_k^{i}) \vert_{_2}}  \Bigg), \qquad 
\end{eqnarray}
where $m_k^{\ell}(x) = x + b(t_k,x) \Delta$ and $\vert   \sigma_k^{\ell {\scriptscriptstyle \bullet}}(x) \vert_{_2}$ is the Euclidean norm of the $\ell$-th row of the volatility matrix $\sigma(t_k,x)$, for $x \in \mathbb R^d$. We deduce immediately the formulas for the  probabilities $\mathbb P(\tilde X_{k+1}^{\ell} \in C_{j_{\ell}}(\Gamma_{k+1}^{\ell}))$, $k=0, \ldots, n-1$,  $j_{\ell} \in \{1, \ldots, N_{k+1}^{\ell} \}$ using  \eqref{EqProbaCondComp_iIntro} (and \eqref{EqWeightsIntro}). 

Equation \eqref{EqProbaCondComp_iIntro}  allows us to approximate the price of the call in the Heston model by  
\begin{equation}  \label{EqApproxHestonCall}
 e^{-rT} \sum_{j_1 = 1}^{N_n^1} H(x_n^{1j_1}) \, \mathbb P(\hat X_n^1 = x_n^{1j_1}).
 \end{equation}
 
Another important issue form the analytical point of view is to compute the quantization error  bound associated to the Markovian quantization process.  Using some results from \cite{PagSagMQ}, we show  (in particular, when  $N_k^{\ell} =N_k$, for avery $\ell=1, \ldots,d$) that  for any sequence $(\hat X_{k}^{\Gamma_k})_{0 \le k \le n}$ of (quadratic) Markovian product quantization  of  $(\tilde X_{k})_{0 \le k \le n}$, the quantization error  $\Vert \bar X_{k} -  \hat X_{k}^{\Gamma_k}  \Vert_{_2}$, at  step $k$ of the recursion,   is bounded by the cumulative quantization errors  $ \Vert \tilde X_{k'} -  \hat X_{k'}^{\Gamma_{k'}}  \Vert_{_2}$, for $k'=0, \ldots, k$. More precisely, one shows that  for  every $k=0, \ldots, n$, for   any $\eta\!\in (0,1]$,
 \begin{equation*} 
  \Vert  \bar X_k -  \hat X_k^{\Gamma_k} \Vert_{_2}  \le     \sum_{k'=0}^{k}  a_{k'} (b,\sigma,  k, d, \Delta, x_0,\eta)  N_{k'} ^{-1/d},
 \end{equation*}
where $a_{k'}(b,\sigma,k, d, \Delta, x_0,\eta)$ is a positive real constant depending on $b$, $\sigma$, $\Delta$, $x_0$, $\eta$ (see Theorem~\ref{TheoConvergenceRate} further on for a more precise statement).


The paper is organized as follows: in Section \ref{SecOptiQuant} we recall some basic results on optimal quantization. Section \ref{SectMQ} is the main part of this paper. We present the algorithm  and  show the Markov property  of the product quantization of the Euler scheme of a diffusion process. Then, we show how to compute the weights and transition probabilities  associated to the product quantizers and to its components.  We also show how to compute  the optimal quantizers associated to each component  of the Euler scheme (keep in mind that this is the foundation of our method). Finally, we provide  in  Theorem~\ref{TheoConvergenceRate} some a priori error bounds for the quantization error associated to the Markovian product quantization and show that, at every step discretization step  $t_k$, this error is a cumulated (weighted) sum of the regular quantization errors, up to time $t_k$.  In Section \ref{SectNum}, we present some numerical results for the pricing of a European call Basket option and a European call option in the Heston model, as well for the approximation of BSDEs.

\bigskip 
\noindent {\sc Notations}.  We denote by ${\cal M}(d,q,\mathbb R)$, the set of $d \times q$ real value matrices.  If $A =[a_{ij}] \in {\cal M}(d,q,\mathbb R)$,  $A^{\star}$ denotes its  transpose  and we define the norm   $\Vert  A \Vert := \sqrt{{\rm Tr}(AA^{\star})} = (\sum_{i,j} a_{ij}^2)^{1/2}$, where ${\rm Tr}(M)$ stands for the trace of $M$, for $M \in {\cal M}(d,d,\mathbb R)$. For every $f:\mathbb R^d \to {\cal M}(d,q,\mathbb R) $,  we will set $[f]_{\rm Lip}= \sup_{x\neq y}\frac{\Vert f(x)-f(y) \Vert}{\vert x-y \vert}$.  For $x, y \in \mathbb R$, $x \vee y  = \max(x,y) $. If $x \in \mathbb R^d$, $\vert x \vert_{_2}$ will stand for the Euclidean norm on $\mathbb R^d$.  For every vectors $x, y$, the notation  $\prods{x}{y}$ denotes the dot product of the vectors $x$ and $y$. For a given vector $z \in \mathbb R^q$ and a matrix $M \in {\cal M}(d,q,\mathbb R)$, $z^i$ denotes the component $i$ of $z$,  $z^{(j:k)}$  the vector  made up from  the component $j$ to the component $k$ of the vector  $z$ and $M^{(i,j:k)}$ is the vector  made up from  the component $j$ to the component $k$ of the  $i$-th row of  the matrix $M$ and $M^{ij}$ for the component  $(i,j)$ of the matrix $M$. The notation $M^{i{\scriptscriptstyle \bullet}}$ stands for the $i$-th row of $M$. The function $\Phi_0$ will denote the cumulative distribution function of the standard real valued Normal distribution and  its derivative $\Phi_0'$ will stand for its density function.

\section{Brief background  on optimal quantization}  \label{SecOptiQuant}
 Let $(\Omega,\mathcal{A},\mathbb{P})$ be a  probability space and  let   $X : (\Omega,\mathcal{A},\mathbb{P}) \longrightarrow  \mathbb{R}^{d} $ be a random variable with  distribution $\mathbb{P}_X$.  The $L^r$-optimal quantization  problem at level  $N$ for the random vector $X$ (or for  the distribution $\mathbb P_X$)  consists in finding  the best approximation of  $X$  by a Borel function  $\pi(X)$ of  $X$  taking  at most  $N$ values. Assuming that $X \in L^r(\mathbb{P})$, we associate to every Borel function $\pi(X)$ taking at most $N$ values,  the $L^r$-mean  error  $\Vert  X - \pi(X)\Vert_r$ measuring the distance between the two random vectors $X$ and $\pi(X)$ w.r.t. the mean  $L^r$-norm,  
 where ${\Vert X \Vert}_r := {\left(\mathbb E \vert X \vert ^r \right)}^{1/r}$ and  $ \vert \cdot \vert $ denotes an Enclidean   norm on $\mathbb{R}^d$. Then finding the best approximation of $X$  by a Borel function   of  $X$  taking  at most  $N$ values  turns out to solve  the following minimization problem: 
$$ e_{N,r}(X) = \inf{\{ \Vert X - \pi(X) \Vert_{r}, \pi: \mathbb{R}^d \rightarrow  \Gamma, \Gamma \subset \mathbb{R}^d,  \vert \Gamma \vert \leq N \}},$$
where $\vert A \vert $ stands for the cardinality of $A$, for $A \subset \mathbb R^d$.  Now, let $\Gamma = \{x_1,\cdots,x_N  \}\subset \mathbb{R}^d$ be a codebook of size $N$ (also called  an $N$-quantizer or a grid of size $N$) and define a Voronoi partition  $C_i(\Gamma)_{ i=1,\cdots,N}$ of $\mathbb{R}^d$, which is a Borel partition of  $\mathbb R^d$  satisfying for every $i \in \{1,\cdots,N\}$,  $$ C_i(\Gamma) \subset  \big\{ x \in \mathbb{R}^d : \vert x-x_i \vert = \min_{j=1,\cdots,N}\vert x-x_j \vert \big\}.$$ 
 Consider the Voronoi quantization  of $X$ (or simply  quantization of $X$) by the $N$-quantizer $\Gamma$ defined by 

$$ 
\hat{X}^{\Gamma} = \sum_{i=1}^N  x_i  \mathds{1}_{\{X \in C_i(\Gamma)\}} . 
$$ 
 
Then, for any Borel function $\pi: \mathbb{R}^d \rightarrow \Gamma= \{x_1,\cdots,x_N \}$ we have
$$ \vert X -\pi(X) \vert \geq \min_{i=1,\cdots,N} d(X,x_i) = d(X,\Gamma)=\vert X - \hat{X}^{\Gamma} \vert  \quad \mathbb{P} \textrm{  a.s}  $$  so that the optimal  $L^r$-mean quantization error  $e_{N,r}(X) $ reads
\begin{eqnarray}
 e_{N,r}(X)  & = & \inf{\{ \Vert X - \hat{X}^{\Gamma} \Vert_r, \Gamma \subset \mathbb{R}^d,  \vert \Gamma \vert  \leq N \}}  \nonumber \\
  & = &  \inf_{ \substack{\Gamma  \subset \mathbb{R}^d \\ \vert \Gamma \vert  \leq N}} \left(\int_{\mathbb{R}^d} d(z,\Gamma)^r d\mathbb P_X(z) \right)^{1/r}. \label{er.quant}
 \end{eqnarray}
 
 Recall that  for every  $N \geq 1$, the infimum in $(\ref{er.quant})$ is attained  at  least  one  codebook. Any $N$-quantizer realizing this infimum  is called an  $L^r$-optimal $N$-quantizer. Moreover, when   $ \vert {\rm supp}(\mathbb P_X))  \vert \geq N$ then  any $L^r$-mean  optimal $N$-quantizer  has exactly size   $N$ (see \cite{GraLus} or \cite{Pag}). On the other hand,    the   quantization error, $e_{N,r}(X)$, decreases to zero as the grid size $N$ goes to infinity and its rate of convergence is ruled by the so-called Zador Theorem  recalled  below.  There also  is a non-asymptotic upper bound for optimal quantizers. It is called Pierce Lemma (we recall it below for the quadratic case) and will allows us  to put a finishing  touches to the  proof of the main result of the paper, stated in  Theorem \ref{TheoConvergenceRate}.

 \begin{thm}  \label{ThmPierceZad}  (a)  Sharp asymptotic rate ({\bf Zador} Theorem, see \cite{GraLus, Zad}).  Let $X$ be an $\mathbb R^d$-valued random  vector  such that  $\mathbb{E}\vert X \vert^{r+ \eta} < + \infty \textrm{ for some } \eta >0 $ and  let $\mathbb P_X= f \cdot \lambda_d + P_s$ be the Lebesgue decomposition of $\mathbb P_X$ with respect to the Lebesgue measure $\lambda_d$  and $P_s$ denotes its singular   part. Then 
 \begin{equation}
  \lim_{N \rightarrow +\infty} N^{\frac{1}{d}} e_{N,r}(P) = \tilde Q_r(\mathbb P_X)
  \end{equation}
where $$ \tilde Q_r(\mathbb P_X) = \tilde J_{r,d} \left( \int_{\mathbb{R}^d} f^{\frac{d}{d+r}} d\lambda_d \right)^{\frac{1}{r} +\frac{1}{d}} = \tilde J_{r,d} \ \Vert f \Vert_{\frac{d}{d+r}}^{1/r} \ \in [0,+\infty) $$
 $$ 
 \tilde J_{r,d} = \inf_{N \geq 1} N^{ \frac{1}{d}} e_{N,r} (U([0,1]^d)) \in (0,+\infty)
 $$
with $ U([0,1]^d) $ denotes the uniform distribution over the hypercube $[0,1]^d$.  

\noindent (b) Non-asymptotic bound  ({\bf Pierce} Lemma, see \cite{GraLus, LusPag2}).   Let $\eta>0$. There exists  a universal constant  $K_{2,d,\eta}$  such that  for every   random vector $X:(\Omega,{\cal A}, \mathbb P) \rightarrow \mathbb R^d$,  
 \begin{equation}  \label{LemPierce} 
\inf_{\vert  \Gamma  \vert  \leq N} \Vert   X   - \hat X^{\Gamma}  \Vert_{_2}   \le  K_{2, d,\eta} \, \sigma_{2,\eta}(X) N^{-\frac{1}{d}}
 \end{equation} 
 where
 $$  \sigma_{2,\eta}(X) = \inf_{\zeta \in \mathbb R^d} \Vert X-\zeta  \Vert_{_{2+\eta}} \le +\infty.$$
\end{thm}

 From the Numerical Probability point of view, finding an optimal $N$-quantizer  $\Gamma$ may be a challenging task. In practice (we will only consider the quadratic case, i.e. $r=2$ for numerical implementations)  we are sometimes led to find some ``good'' quantizations  $\hat X^{\Gamma}$ which are close to $X$ in distribution, so that  for every continuous  function $f: \mathbb R^d \rightarrow  \mathbb R$, we can approximate $\mathbb E f(X)$ by  
\begin{equation} \label{QuantProcedureEstim}
\mathbb{E}  f \big(\hat{X}^{\Gamma} \big) = \sum_{i=1}^N p_i f(x_i)
\end{equation}
where $p_i =\mathbb{P}( \hat{X}^{\Gamma} = x_i ).$  When we approximate $\mathbb E f(X)$ by \eqref{QuantProcedureEstim}, this induced an error  which bound depends on the regularity of the function $f$ (see e.g. \cite{PagPri03} for more details).    

We recall below  the stationarity property for a  quantizer.

\begin{defn} A quantizer  $\Gamma = \{ x_1,\ldots,x_N \}$  of size $N$ inducing the Voronoi quantization $\hat{X}^{\Gamma}$  of $X$  is  stationary  if  
 $\mathbb{P}\left(X\!\in \cup_{i} \partial  C_i(\Gamma) \right) =0$,   $\mathbb{P}\left(X\!\in   C_i(\Gamma) \right) >0$, $i=1,\ldots,N$,  and
\begin{equation}
   \mathbb{E} \big(X  \vert  \hat{X}^{\Gamma} \big) =\hat{X}^{\Gamma} \; \mathbb P\mbox{-}a.s.   \quad \Longleftrightarrow  \quad x_i = \frac{\mathbb E(X\mathbb {1}_{\{X\!\in C_i(\Gamma)\}})}{\mathbb P(X\!\in C_i(\Gamma))},\; i=1,\ldots,N.
\end{equation}
 \end{defn}
 The notion of stationarity is  related to the critical point of the so-called {\em distortion} function defined on $(\mathbb R^d)^N$  by  
\begin{equation}\label{EqDistor}
 D_{N,2} (x)  =  \mathbb E \big(\min_{1\le i\le N}|X-x_i|^2\big)= \int_{\mathbb R^d}|\xi-x_i|^2\mathbb P_{X}(d\xi), \quad x=(x_1,\ldots,x_{_N})\!\in (\mathbb R^d)^N.
  \end{equation}
As any grid of size at most $N$ can be ``represented'' by some $N$-tuples (by repeating, if necessary, some of its  elements), we will often put grids of all size $N$ as an argument of the distortion function $D_{2,N}$ as well as for its gradient and Hessian matrix  when its Voronoi boundary is negligible.  It is also clear, from the definition of the quantization error, that  
 \[
 e_{N,2}^2(X) =\!\inf _{(x_1,\ldots,x_N)\in (\mathbb R^d)^N} D_{N,2}(x_1,\ldots,x_N).
 \]
 Furthermore, the function $D_{N,2}$ is continuous and  differentiable at any $N$-tuple having pairwise distinct components with a $\mathbb P$-negligible Voronoi partition boundary and the following result  makes this more precise. 
 \begin{prop} \label{PropDifferentiability}(see~\cite{GraLus, Pag})
 $(a)$ The function $D_{N,2}$ is differentiable at any $N$-tuple $(x_1,\ldots,x_{_N}) \!\in (\mathbb R^{d})^N$ having pairwise distinct components   and such that $\mathbb{P}\left(X\!\in \cup_{i} \partial  C_i(\Gamma) \right)=0$. Furthermore, we have
 \begin{eqnarray}
  \nabla D_{N,2}(x_1,\ldots,x_{_N})  &= & 2 \Big(   \!\int_{C_i(\Gamma ) }  (x_i - x ) d \mathbb P_X (x)  \Big)_{i=1,\ldots,N}\\
  &=& 2 \Big(\mathbb P(X\!\in C_i(\Gamma))x_i-\mathbb E(X\mathds{1}_{\{X\!\in C_i(\Gamma)\}}) \Big)_{i=1,\ldots,N}.
 \end{eqnarray}
  
  \smallskip
  \noindent $(b)$ A grid $\Gamma= \{x_1,\ldots,x_{_N}\}$ of full size $N$ is stationary if and only if 
  \begin{equation}\label{eq11}
  \mathbb{P}\left(X\!\in \cup_{i} \partial  C_i(\Gamma) \right) =0 \quad \mbox{ and }\quad    \nabla D_{N,2} (\Gamma) = 0.
\end{equation}

\smallskip 
  \noindent   $(c)$  If the support of $\mathbb P_{_X}$ has at least $N$ elements, any $L^2$-optimal quantizer at level $N$ has full size and a $\mathbb P$-negligible Voronoi boundary. Hence it is a  stationary $N$-quantizer.  
\end{prop}

For numerical  implementations, the search of stationary quantizers  is based on zero search recursive procedures like Newton-Raphson    algorithm for real valued random variables,  and  some algorithms like  Lloyd's I algorithms (see e.g. \cite{GerGra,PagYu}),   the Competitive Learning Vector Quantization (CLVQ) algorithm (see \cite{GerGra}) or stochastic algorithms (see \cite{PagPri03}) in the multidimensional framework.  Optimal  quantization grids associated to  multivariate Gaussian random vectors can  be downloaded  on the  website {\tt www.quantize.math-fi.com}.

\section{Markovian product quantization of an $\mathbb R^d$-valued   Euler process}  \label{SectMQ}
Let $(X_t)_{t \geq 0}$ be a stochastic process taking values in a  $d$-dimensional Euclidean space $\mathbb R^d$ and  solution to  the stochastic differential equation: 
\begin{equation} \label{EqSignalProcess}
 X_t = x_0 + \int_0^t b(s,X_s) ds + \int_0^t\sigma(s,X_s) dW_s,   \quad   x_0 \in \mathbb R^d, 
 \end{equation}
where  $W$ is a  standard  $q$-dimensional  Brownian motion starting  at $0$ and where   $b:[0,T] \times \mathbb R^d \rightarrow  \mathbb R^d$ and the matrix diffusion coefficient function  $\sigma:[0,T] \times \mathbb R^d \rightarrow  {\cal M}(d,q, \mathbb R) $ are measurable and satisfy the global Lipschitz  continuity and linear growth conditions: for every $t \in [0,T]$,
\begin{eqnarray}
& & \vert b(t,x) - b(t,y) \vert    \le  [b]_{\rm Lip} \vert  x-y  \vert    \label{LipAssb} \\ 
&  &  \Vert  \sigma(t,x) - \sigma(t,y)\Vert  \leq  [\sigma]_{\rm Lip} \vert  x-y  \vert  \label{LipAssSig}  \\ 
 &  & \vert b(t,x)   \vert  \leq L (1 +\vert x \vert)   \  \textrm{  and  } \  \Vert \sigma(t,x)   \Vert  \leq  L  (1 +\vert x \vert) \label{EqLinGrowth}
\end{eqnarray}  
for some  $L>0$.  This guarantees the existence and pathwise uniqueness  of  a strong solution of (\ref{EqSignalProcess}), adapted to the (augmented) filtration of $W$. We also suppose that  the matrix $\sigma$  is  positive definite. Throughout  the paper we will suppose that  $\mathbb R^d$ is equipped with the canonical  Euclidean norm.

\subsection{The algorithm and the Markov property of the quantized process}

Recall that the  Euler scheme of the stochastic process $(X_t)_{t \geq 0}$ is defined recursively from the following procedure:
\[
\bar X_{t_{k+1}} = \bar X_{t_k} + \Delta  b(t_k,\bar X_{t_k})  + \sigma(t_k,\bar X_{t_k}) (W_{t_{k+1}} - W_{t_k}), \quad \bar X_{0}  \in \mathbb R^d,
\]
where $\Delta=\Delta_n = \frac{T}{n}$ and  $t_k =\frac{k T}{n}$, for every $k \in \{ 0, \cdots,n \}$.  To simplify  notations, we  will often set   $X_k:= X_{t_k}$  to denote the process $X$ evaluated at time  $t_k$. We also set  $b_k(x):= b(t_k,x)$  and $\sigma_k(x) = \sigma(t_k,x)$ for  $ x \in \mathbb R^d$.  Recall also that the Euler operator associated to the conditional distribution of $\bar X_{k+1}$ given $\bar X_k =x$ is defined by  
  \[
  {\cal E}_k(x,z) := x+\Delta b(t_k,x)+\sqrt{\Delta} \sigma(t_k,x) z,\quad  x\!\in \mathbb R^d, \; z \!\in \mathbb R^q
  \]
and that if $\Gamma_{k+1}$ is an $N_{k+1}$-quantizer for $\bar X_{k+1}$, the distortion function  $\bar D_{k+1}$ associated to $\bar  X_{k+1}$  may be written   for every $k=0, \cdots,n-1$, as
   \begin{eqnarray*}
   \bar  D_{k+1}(\Gamma_{k+1}) & = &  \mathbb E \big(({\rm dist}(\bar X_{k+1} ,\Gamma_{k+1})^2 \big) \\
     &=& \mathbb E \big[  {\rm  dist} ({\cal E}_k(\bar X_k,Z_{k+1}), \Gamma_{k+1})^2 \big] \label {EqDistorNonQuant}
  \end{eqnarray*}
  where $Z_{k+1} \sim \mathcal{N}(0;I_q)$ is independent from $\bar X_k$. The previous way to write the distortion function has been used in \cite{PagSagMQ} to propose a fast recursive (and Markovian) quantization of the Euler process (using the Newton-Raphson algorithm for the numerical computation of the optimal grids) when $d=1$.  
  
  Keep in mind that the conditional distribution of the discrete Euler process $\bar X$ is Gaussian and that one of the properties of a Gaussian vector is that any sub-component  of the vector remains a  Gaussian  random vector. So, a natural alternative way to quantize the vector $\bar X_k \in \mathbb R^d$ is to quantize each component $\bar X_k^{\ell}$ by  a grid $\Gamma_k^{\ell}$ of size $N_k^{\ell}$, for $\ell=1, \ldots, d$,  and then to define its  product  quantization  $\hat X_k$ associated  with  the product quantizer  $\Gamma_k = \bigotimes_{\ell=1}^d  \Gamma_k^{\ell}$ of size  $N_k = N_k^1  {\small \times }  \ldots {\small \times }  N_{k}^d$,     as  $\hat X_k =(\hat X_k^1, \ldots, \hat X_k^d)$. 
  
  The question  is now to know how to quantize the  $\bar X_k^{i}$'s. On the other hand, since the components of the vector $\bar X_k$ are not independent it is also a challenging question to know  how to compute  (from closed formula) the companions weights and transition  probabilities associated with the quantizations of the $\bar X_k^{i}$'s and the vector $\bar X_k$.  We describe below the componentwise recursive Markovian  quantization of the process $\{\bar X_k, \ k=0, \ldots,n\}$.  
 \vskip 0.2cm
 
  It is clear that for every $\ell =1, \ldots,d$, and for every $k=0, \ldots, n-1$,  the transition operator ${\cal E}_k^{\ell}(x,z) $  associated with the  distribution of $\bar X_{k+1}^{\ell}$ given $\bar X_k =x$ reads as
   \[
  {\cal E}_k^{\ell}(x,z) := m_k^{\ell}(x)+\sqrt{\Delta} \big(\sigma_k^{\ell {\scriptscriptstyle \bullet}} (x)  \vert z\big),\quad  x\!\in \mathbb R^d, \; z \!\in \mathbb R^q,
  \]
  where
  \[
  m_k^{\ell}(x):=  x^\ell+\Delta b_k (x) .
  \]
 
For every $k =0, \ldots,n$, for every given  $\ell \in \{ 1, \ldots,d \}$, we denote by $\hat X_k^{\ell}$  the quantization of $\bar X_k^{\ell}$ on the  grid $\Gamma_k^{\ell} = \{ x_k^{\ell,i_{\ell}},  i_{\ell} =1, \ldots, N_k^{\ell} \}$. 
  We propose in what follows  a recursive and  componentwise product   quantization  of the process  $\{\bar X_k, \ k=0, \ldots,n \}$. In fact, for every $\ell=1, \ldots,d$, we denote by $\Gamma_k^{\ell}$ an  $N_k^{\ell}$-quantizer (we suppose that we have access to it) of   the $\ell$-th component  $\bar X_k^{\ell}$ of the vector $\bar X_k$ and  by $\hat X_k^{\ell}$, the  resulting quantization  of $\bar X_k^{\ell}$.    Then, we define a componentwise recursive product quantizer $\Gamma_k =  \bigotimes_{\ell=1}^d  \Gamma_k^{\ell}$ of size  $N_k = N_k^1  {\small \times }  \ldots {\small \times }  N_{k}^d$ of  the vector $\bar X_k = (\bar X_k^{\ell})_{\ell=1, \ldots,d}$  by   
  \begin{eqnarray*}
  \Gamma_k &  = & \big \{ (x_k^{1,i_1}, \ldots,x_k^{d,i_d}),   \quad i_{\ell} \in \big\{1, \ldots,N_k^{\ell} \big \}, \  \ell \in \{1, \ldots,d\}  \big\}.  
  \end{eqnarray*}

To define the Markovian product quantization,  suppose  that  $\bar  X_k$ has already been quantized and that we have access to the companion weights $\mathbb P(\hat X_k = x_k^{i})$,      $ i \in I_k$, where $I_k$ and $x_k^{i}$ are defined by Equations  \eqref{Eqmultindex} and \eqref{EqmultindexComp}.   Setting $\tilde X_{k}^{\ell}  =  {\cal E}_k^{\ell}(\hat X_k ,Z_{k+1})$, we may approximate the distortion function $\bar D_{k+1}^{\ell}$ associated to the $\ell$-th component of the vector $\bar X_{k+1}^{\ell}$  by
     \begin{eqnarray*}
  \tilde D_{k+1}^{\ell} (\Gamma_{k+1}^{\ell}) & := &   \mathbb E \big[   {\rm dist}(\tilde X_{k+1}^{\ell}, \Gamma_{k+1}^{\ell})^2 \big] \\
  & = & \mathbb E \big[  {\rm  dist}({\cal E}_k^{\ell}(\hat X_k ,Z_{k+1}), \Gamma_{k+1}^{\ell})^2 \big] \\ 
  &  = &  \sum_{i \in I_k}  \mathbb E \big[   {\rm dist}({\cal E}_k^{\ell}(x_k^{i},Z_{k+1}), \Gamma_{k+1})^2 \big] \mathbb P\big(\hat X_k =x_k^{i} \big). 
 \end{eqnarray*}
 
 This allows us  to consider the sequence of  product recursive  quantizations of  $(\hat X_k)_{k=0,\cdots,n}$, defined for every $k=0, \ldots,n-1$,  by the following recursion: 
%
 \begin{equation}  \label{EqAlgorithm}
   \left \{ \begin{array}{l}
 \tilde X_0 = \hat X_0, \quad  \hat X_k ^{\ell} = {\rm Proj}_{\Gamma_k^{\ell}}(\tilde X_k^{\ell}), \ \ell=1, \ldots, d,   \\
   \hat X_k  =  ( \hat X_k^1, \ldots, \hat X_k^d)  \quad  \textrm{and} \quad  \tilde X_{k+1}^{\ell}  =  {\cal E}_k^{\ell}(\hat X_k,Z_{k+1}), \ \ell=1, \ldots, d, \\
   {\cal E}_k^{\ell} (x,z)= m_k^{\ell}(x) + \sqrt{\Delta} (\sigma^{\ell \bullet}(t_k,x) \vert  z),  \ z = (z^1, \ldots,z^q)\in \mathbb R^q,\\
     x=(x^1, \ldots,x^d), \ b=(b^1, \ldots, b^d)   \textrm{ and } (\sigma^{\ell \bullet}(t_k,x) \vert  z) =\sum_{m=1}^q \sigma^{\ell  m}(t_k,x) z^m. 
 \end{array}  
 \right.
 \end{equation}
 where   $(Z_k)_{k=1,\cdots, n}$ is   $i.i.d.$,  ${\cal N}(0;I_q)$-distributed, independent of   $\bar X_0$.

In the following result, we show   that the sequence $(\hat X_k)_{k \ge 0}$ of  Markovian  and product quantizations is in fact a Markov chain.  Its transition probabilities will be computed further on. 
\medskip 
\begin{rem} \label{PropMarkovChainProper}
We may remark that the process $(\hat X_k)_{k \ge 0}$ is a Markov chain on $\mathbb R^d$.

In fact, setting ${\cal F}_ k^{\hat X}  = \sigma(\hat X_0, \ldots, \hat X_k)$, we have  for any bounded function $f: \mathbb R^d \rightarrow \mathbb R$
\begin{eqnarray*}
\mathbb E(f(\hat X_{k+1})  \vert {\cal F}_ k^{\hat X}) &=& \sum_{j \in I_{k+1}} \mathbb E \left(f(x_{k+1}^j)  \mathds{1}_{ \{ \hat X_{k+1} =x_{k+1}^j  \}} \vert  {\cal F}_ k^{\hat X} \right) \\
& = & \sum_{j \in I_{k+1}} f(x_{k+1}^j)   \mathbb E \left( \mathds{1}_{ \{ \mathcal E_k( \hat X_{k},Z_{k+1}) \in \prod_{\ell=1}^d C_{j_{\ell}}(\Gamma_{k+1}^{\ell})  \}} \vert {\cal F}_ k^{\hat X} \right),
\end{eqnarray*}
where $\mathcal{E}_k(\hat X_k,Z_{k+1})= (\mathcal E_k^1(\hat X_k,Z_{k+1}), \ldots, \mathcal E_k^d(\hat X_k,Z_{k+1}))$.  It follows that
\[
\mathbb E(f(\hat X_{k+1})  \vert {\cal F}_ k^{\hat X}) = \sum_{j \in I_{k+1}} f(x_{k+1}^j)  h_j(\hat X_k),
\]
where for every $x \in \mathbb R^d$,   
\[
h_j (x)  =  \mathbb P \big(  \mathcal E_k(x,Z_{k+1}) \in \prod_{\ell=1}^d C_{j_{\ell}}(\Gamma_{k+1}^{\ell})   \big).
\]
As a consequence, $\mathbb E(f(\hat X_{k+1})  \vert {\cal F}_ k^{\hat X})  = \varphi(\hat X_k)$, so that $(\hat X_k)_{k \ge 0}$  is a Markov chain.
\end{rem}

 Now, for a given componentwise (quadratic)  optimal quantizers $\Gamma_k  = \bigotimes_{\ell=1}^d \Gamma_k^{\ell}$, let us explain   how to compute the   companion  transition  probability  weights associated with the quantizations of the $\bar X_k^{\ell}$'s and the whole  vector $\bar X_k$.  We write all the quantities of interest as an expectation of a function of a standard  $\mathbb R^{q-1}$-valued Normal distribution.  These transformations are the key step of this work. In fact, since the optimal  quantization grids associated to  standard Normal random vectors (up to dimension $10$) and their companion weights  are available  on  {\tt www.quantize.maths-fi.com},  these quantities of interest may be computed instantaneously using a cubature formula.

 \subsection{Computing the companion weights and transition probabilities of the marginal quantizations}  \label{SubsectionWeights}

 First of all we define the following quantities which will be  needed in the sequel. For every $k \in \{0, \ldots, n-1 \}$ and  for every $j \in I_{k+1}$  we set  
\[
 x_{k+1}^{i, j_i-1/2} = \frac{ x_{k+1}^{i,j_i} + x_{k+1}^{i, j_i-1}  }{2}, \  x_{k+1}^{i,j_i+1/2} = \frac{x_{k+1}^{i,j_i} + x_{k+1}^{i, j_i+1}  }{2}, \ \textrm{ with }  x_{k+1}^{i,1/2} = -\infty,  x_{k+1}^{i, N^i_{k+1}+1/2} =+\infty, 
 \]
and if $Z_k^{(2:q)} = z \in \mathbb R^{q-1}$ and $ x \in \mathbb R^d$,  we set  (if $\sigma_k^{i1}(x) >0$)
\begin{eqnarray*}
& &  x_{k+1}^{i,j_i-}(x,z): = \frac{x_{k+1}^{i,j_i-1/2} - m_{k}^i(x) - \sqrt{\Delta} \big(\sigma_k^{(i,2:q)}(x) \vert z \big) }{\sqrt{\Delta} \sigma_k^{i1}(x)  } \\
& \textrm{ and } & x_{k+1}^{i,j_i+}(x,z): = \frac{x_{k+1}^{i,j_i+1/2} - m_{k}^i(x) - \sqrt{\Delta} \big(\sigma_k^{(i,2:q)}(x) \vert z \big) }{\sqrt{\Delta} \sigma_k^{i1}(x)}.
 \end{eqnarray*}  
We also set
\[
\mathbb{J}^0_{k,j_i}(x) = \Big\{ z \in \mathbb R^{q-1}, \quad \sqrt{\Delta} \big(\sigma_k^{(i,2:q)}(x) \big \vert z \big)  \in \big(  x_{k+1}^{i,j_i-1/2} - m_{k}^i(x), x_{k+1}^{i,j_i+1/2} - m_{k}^i(x)  \big) \Big \}
\] 
and
\begin{eqnarray*}
&  & \mathbb{J}^0_{k}(x) = \big\{ i \in \{1, \ldots, d \}, \quad  \sigma_k^{i1} (x) = 0 \big \} \\
& & \mathbb{J}^{-}_{k}(x) = \big\{ i \in \{1, \ldots, d \}, \quad  \sigma_k^{i1} (x) < 0 \big \} \\
&  & \mathbb{J}^{+}_{k}(x) = \big\{ i \in \{1, \ldots, d \}, \quad  \sigma_k^{i1} (x) > 0 \big \}.
\end{eqnarray*}

The following result allows us to compute the weights and the transition probabilities  associated to the quantizations $\widehat X_k$, $k=0, \ldots,n$. 

\begin{prop}  \label{PropTransitionProba}
Let  $\{\widehat X_{k}, k=0, \ldots,n  \}$ be the sequence defined from the algorithm  \eqref{EqAlgorithm}.

  The transition probability $\mathbb P(\widehat X_{k+1} = x_{k+1}^{j} \vert \widehat X_{k}  =x_k^{\ell})$, $i \in I_k$,  $j  \in I_{k+1}$, is given   by 
  \setlength\arraycolsep{1pt}
\begin{eqnarray} \label{EqProbaCondVector}
 \mathbb P\big(\widehat X_{k+1} = x_{k+1}^{j} \vert \widehat X_{k}  = x_k^{i} \big)  & = &  \mathbb E \prod_{ \ell \in \mathbb J^{0}_k(x^{i}_k)} \mbox{\bf{1}}_{\{  \zeta \in \mathbb J^0_{k,j_{\ell}}(x^{i}_k)  \}} \max\big( \Phi_0(\beta_{j} (x_k^{i},\zeta))  -   \Phi_0(\alpha_{j}(x_k^{i},\zeta)),0\big)   \qquad
\end{eqnarray}
where $\zeta \sim {\cal N}(0; I_{q-1})$ and, for every $x \in \mathbb R^d$ and $z \in \mathbb R^{q-1}$,
\begin{eqnarray*}
& & \alpha_{j}(x,z)= \max  \big( \sup_{\ell \in \mathbb J^{+}_{k} (x)} x_{k+1}^{\ell,j_{\ell}-} (x,z),  \sup_{\ell \in \mathbb J^{-}_{k} (x)}  x_{k+1}^{\ell,j_{\ell}+} (x,z)\big) \\
& \textrm{ and } & \beta_{j}(x,z)= \min  \big( \inf_{\ell \in \mathbb J^{+}_{k} (x)} x_{k+1}^{\ell,j_{\ell}+} (x,z),  \inf_{\ell \in \mathbb J^{-}_{k} (x)}  x_{k+1}^{\ell,j_{\ell}-} (x,z)\big).
\end{eqnarray*}

\end{prop}

 Before proving this result, remark that we may deduce the probability weights associated to the quantizations $(\hat X_{k+1})$  by
\setlength\arraycolsep{1pt}
\begin{eqnarray} 
\mathbb P \big(\hat X_{k+1}  = x_{k+1}^{j}  \big)  & =&  \sum_{i \in I_k}   \mathbb P\big(\hat X_{k+1} = x_{k+1}^{j} \vert \hat X_{k}  = x_k^{i} \big)     \mathbb P \big(\hat X_{k} = x_k^{i})  \label{EqProbaVector}
\end{eqnarray}
where the conditional probabilities  are computed using the formula \eqref{EqProbaCondVector}.

\begin{proof} 
  We have
\begin{eqnarray*}
 \mathbb P(\widehat X_{k+1} = x_{k+1}^{j} \vert \widehat X_{k}  =x_k^{\ell}) & = &   \mathbb P \Big(\bigcap_{i=1}^d \big\{\widetilde X_{k+1}^i \in \big(v^{i,j_i-}, v^{i,j_i+} \big) \big\}   \big\vert \widehat X_{k}  =x_k^{\ell} \Big) \\
 & = &    \mathbb P \Big(\bigcap_{i=1}^d \big\{{\cal E}_k^i(x_k^{\ell},Z_{k+1}) \in \big(v^{i,j_i-}, v^{i,j_i+} \big)  \big\}  \Big) \\
 & = &   \mathbb E\Big( \mathbb E \Big(  \mbox{\bf{1}}_{\bigcap_{i=1}^d \big\{{\cal E}_k^i(x_k^{\ell},Z_{k+1}) \in \big(v^{i,j_i-}, v^{i,j_i+} \big) \big\}}  \Big) \big\vert Z_{k+1}^{(2:d)} \Big) \\
 & = &  \mathbb E\big(  \Psi(x_k^{\ell},Z_{k+1}^{(2:d)}) \big)
\end{eqnarray*}
where, for every $u \in \mathbb R^{q-1}$,
\begin{eqnarray*}
\Psi(x,u)  &=&  \mathbb P \Big( \bigcap_{i=1}^d \big\{m_k^i(x_k^{\ell}) + \sqrt{\Delta}\sigma_k^{i1} (x) Z_{k+1}^{1} +  \sqrt{\Delta} \big(\sigma_k^{(i,2:q)}(x) \vert u \big)  \in \big(v^{i,j_i-}, v^{i,j_i+} \big)  \big\}\Big).
\end{eqnarray*}
Let us set 
\[
A_{i,k} = \Big\{m_k^i(x_k^{\ell}) + \sqrt{\Delta}\sigma_k^{i1} (x) Z_{k+1}^{1} +  \sqrt{\Delta} \big(\sigma_k^{(i,2:q)}(x) \vert u \big)  \in \big(v^{i,j_i-}, v^{i,j_i+} \big)  \Big\}.
\]
We know that  if  $i \in \mathbb J^{0}_k(x)$ then   $A_{i,k}  =  \{ u \in  \mathbb J^0_{k,j_i}(x) \}$
and we deduce that
\[
\Psi(x,u)  = \prod_{ i_0 \in \mathbb J^{0}_k(x) } \mbox{\bf{1}}_{\{ u \in  \mathbb J^0_{k,j_i}(x)\} } \mathbb P\Bigg(  \big( \bigcap_{i_{-} \in  \mathbb J^{-}_{k} (x) } A_{i_{-},k} \big) \cap  \big( \bigcap_{i_{+} \in  \mathbb J^{+}_{k} (x) } A_{i_{+},k} \big)  \Bigg).
\]
Furthermore, notice that   if  $i_{+} \in  \mathbb J^{+}_{k} (x)$ then
\[
 A_{i_{+},k}  = \big \{ Z_{k+1}^1  \in   (x_{k+1}^{i,j_i-} (x,u),  x_{k+1}^{i,j_i+} (x,u)) \big\}
\]
and  $i_{-} \in  \mathbb J^{-}_{k} (x)$ then
\[
 A_{i_{-},k}  =  \big\{ Z_{k+1}^1  \in   (x_{k+1}^{i,j_i+} (x,u),  x_{k+1}^{i,j_i-} (x,u)) \big\}.
\]
It follows that  (remark that the sets $\mathbb J^{-}_{k} (x)$ or  $\mathbb J^{+}_{k} (x)$ may be empty)
\begin{eqnarray*}
\mathbb P\Bigg(  \big( \bigcap_{i_{-} \in  \mathbb J^{-}_{k} (x) } A_{i_{-},k} \big) \cap  \big( \bigcap_{i_{+} \in  \mathbb J^{+}_{k} (x) } A_{i_{+},k} \big)  \Bigg)& =&  \mathbb P\Big( Z_{k+1}^1  \in \big( \sup_{i \in  \mathbb J^{+}_{k} (x)}  x_{k+1}^{i,j_i-} (x,u), \inf_{i \in  \mathbb J^{+}_{k} (x)}  x_{k+1}^{i,j_i+} (x,u)  \big) \\
& & \qquad \cap \big( \sup_{i \in  \mathbb J^{-}_{k} (x)}  x_{k+1}^{i,j_i+} (x,u), \inf_{i \in  \mathbb J^{-}_{k} (x)}  x_{k+1}^{i,j_i-} (x,u)  \big) \Big). 
\end{eqnarray*}
This completes the proof since $Z_{k+1}^{(2:d)}  \sim {\cal N}(0; I_{q-1})$.
\end{proof}

Now, we focus on  in the particular case where the matrix $\sigma(t,x)$, for $(t,x) \in [0,T]{\small \times } \mathbb R^d$, is   diagonal  with positive diagonal entries $\sigma^{\ell \ell}(t,x)$, $\ell=1, \ldots, d$. The following result says how to compute the transition probability weights of the $\hat X_k$'s. Let us set for every $x \in \mathbb R^d$, every $\ell \in \{1, \ldots,d\}$ and $j_{\ell} \in \{1, \ldots,N_{k+1}^{\ell} \}$, 
\begin{equation*}
  x_{k+1}^{\ell,j_{\ell}-}(x,0): = \frac{x_{k+1}^{\ell,j_{\ell}-1/2} - m_{k}^{\ell}(x)  }{\sqrt{\Delta} \sigma_k^{\ell \ell}(x)  } \quad \textrm{ and } \quad  x_{k+1}^{\ell,j_{\ell}+}(x,0): = \frac{x_{k+1}^{\ell,j_{\ell}+1/2} - m_{k}^{\ell}(x) }{\sqrt{\Delta} \sigma_k^{\ell \ell}(x)}.
 \end{equation*}

\begin{prop}  \label{PropCasIndep}
Let  $\{\hat X_{k}, k=0, \ldots,n  \}$ be the sequence of quantizers defined by the algorithm  \eqref{EqAlgorithm}  and  associated with the solution $(X_t)$ of \eqref{EqSignalProcess}. Suppose that  the volatility matrix $\sigma(t,x)$  of $(X_t)_{t \ge 0}$  is  diagonal   with positive diagonal entries $\sigma^{\ell \ell}(t,x)$, $\ell=1, \ldots,d$.  Then,
%
the transition probability  weights  $\mathbb P(\hat X_{k+1} = x_{k+1}^{j} \vert \hat X_{k}  =x_k^{i}), \ i \in I_k, j \in I_{k+1}$, are given by
  \setlength\arraycolsep{1pt}
\begin{eqnarray}
 \mathbb P\big(\hat X_{k+1} = x_{k+1}^{j} \vert \hat X_{k}  = x_k^{i} \big) &=& \prod_{\ell=1}^d   \mathbb P\big(\hat X_{k+1}^{\ell} = x_{k+1}^{j_{\ell}} \vert \hat X_{k}  = x_k^{i} \big) \\
  & = & \prod_{\ell=1}^d  \big[\Phi_0 \big( x_{k+1}^{\ell,j_{\ell}+}(x_k^{i},0) \big)  -   \Phi_0 \big( x_{k+1}^{\ell,j_{\ell}-}(x_k^{i},0) \big) \big] \label{EqEstProbaProp2PCond},
\end{eqnarray}
%
and the companion probability weights $\mathbb P \big(\hat X_{k+1}  = x_{k+1}^{j}  \big)$ are given  for every $k=0, \ldots,n-1$  and  every $j \in I_{k+1}$ by
\setlength\arraycolsep{1pt}
\begin{eqnarray}
\mathbb P \big(\hat X_{k+1}  = x_{k+1}^{j}  \big)  & =&  \sum_{i \in I_k} \prod_{\ell=1}^{d}\big[\Phi_0 \big( x_{k+1}^{\ell,j_{\ell}+}(x_k^{i},0) \big)  -   \Phi_0 \big( x_{k+1}^{\ell,j_{\ell}-}(x_k^{i},0) \big) \big]  \mathbb P( \hat X_{k}  = x_k^{i} ). \label{EqEstProbaProp2P}
\end{eqnarray}
\end{prop}

\begin{proof} 
 1.  Set $v^{\ell,j_{\ell}+}:=x_{k+1}^{\ell,j_{\ell}+1/2}$ and $v^{\ell,j_{\ell}-} = x_{k+1}^{\ell,j_{\ell}-1/2}$,  for $j \in I_{k+1}$ and $\ell=1, \ldots, d$.  We have
\begin{eqnarray*}
 \mathbb P(\hat X_{k+1}  =  x_{k+1}^{j} \vert \hat X_{k}  =x_k^{i}) & =&    \mathbb P \Big(\bigcap_{\ell=1}^d \big\{\tilde X_{k+1}^{\ell} \in \big(v^{\ell,j_{\ell}-}, v^{\ell,j_{\ell}+} \big) \big\}   \big\vert \hat X_{k}  =x_k^{i} \Big) \\
 & = & \mathbb P \Big(\bigcap_{\ell=1}^d \big\{{\cal E}_k^{\ell}(x_k^{i},Z_{k+1}) \in \big(v^{\ell,j_{\ell}-}, v^{\ell,j_{\ell}+} \big)  \big\}  \Big).
 \end{eqnarray*}
 Since for every $k=0, \ldots,n-1$, $\sigma(t_k,x)$ is a diagonal matrix, it follows that the operators ${\cal E}_k^{\ell}(x_k^{i},Z_{k+1}) = {\cal E}_k^{\ell}(x_k^{i},Z_{k+1}^\ell)$, for $\ell=1, \ldots,d$, are independent, so that 
  \begin{eqnarray*}
 \mathbb P(\hat X_{k+1}  =  x_{k+1}^{j} \vert \hat X_{k}  =x_k^{i})  & = &   \prod_{\ell=1}^d \mathbb P \Big({\cal E}_k^{\ell}(x_k^{i},Z_{k+1}^{\ell}) \in \big(v^{\ell,j_{\ell}-}, v^{\ell,j_{\ell}+} \big)   \Big) \\
& =& \prod_{\ell=1}^d \big[\Phi_0 \big( x_{k+1}^{\ell,j_{\ell}+}(x_k^{i}) \big)  -   \Phi_0 \big( x_{k+1}^{\ell,j_{\ell}-}(x_k^{i}) \big) \big].
  \end{eqnarray*}
  The second assertion   immediately follows. 
  \end{proof}

The following result is useful in the situation where we need to approximate  the expectation of a function of one component of the vector $\bar X_k$ as for example in the pricing of European options  in the Heston model.

\begin{prop} \label{PropTransComponent}
Let $\Gamma_{k+1}^{\ell}$ be an optimal quantizer for the random variable $\tilde X_{k+1}^{\ell}$. Suppose that  the  optimal product  quantizer $\Gamma_k$ for $\tilde X_{k}$ and its companion  weights $\mathbb P(\hat X_{k} =x_k^{i} )$, $i \in I_k$, are computed.  

 For any $\ell \in \{1, \ldots,d\}$ and any  $j_{\ell} \in \{1, \ldots,N_{k+1}^{\ell} \}$,  the transition probability  weights $\mathbb P(\tilde X_{k+1}^{\ell} \in C_{j_{\ell}}(\Gamma_{k+1}^{\ell})\vert \hat X_{k}  =x_k^{i})$ are given by
  \setlength\arraycolsep{1pt}
\begin{eqnarray} 
 \mathbb P\big(\tilde X_{k+1}^{\ell} \in C_{j_{\ell}}(\Gamma_{k+1}^{\ell})\vert \hat X_{k}  = x_k^{i} \big) 
 & =&  \Phi_0 \big( x_{k+1}^{\ell,j_{\ell}+} (x_k^i) \big) -  \Phi_0 \big( x_{k+1}^{\ell,j_{\ell}-} (x_k^i) \big) \label{EqProbaCondCompos_i}.
\end{eqnarray}
\end{prop}

\begin{proof} 
1. For every $k \in \{ 1,\ldots,n-1 \}$, for every $\ell=1, \ldots,d$ and  for every  $j_{\ell}=1,\ldots,N_{k+1}^{\ell}$, we have
\setlength\arraycolsep{1pt}
 \begin{eqnarray*} \label{EqSumNk-1Dist}
 \mathbb P \big(\tilde X_{k+1}^{\ell} \in C_{j_{\ell}}(\Gamma_{k+1}^{\ell})\vert \hat X_{k} = x_{k}^{i} \big) &=&    \mathbb P \big(\tilde X_{k+1}^{\ell} \leq  x_{k+1}^{\ell, j_{\ell}+1/2}\vert  \hat  X_{k} =  x_k^{i} \big)   -  \mathbb P \big(\tilde X_{k+1} ^{\ell} \leq  x_{k+1}^{\ell,j_{\ell}-1/2} \vert   \hat  X_{k} =  x_k^{i} \big)\\
 & = &  \mathbb P \big(  {\cal E}_k^{\ell}(x_k^{i},Z_{k+1}) \leq  x_{k+1}^{\ell, j_{\ell}+1/2} \big)   -  \mathbb P \big( {\cal E}_k^{\ell}(x_k^{i},Z_{k+1}) \leq  x_{k+1}^{\ell,j_{\ell}-1/2} \big).
 \end{eqnarray*}
To complete the proof we just have to remark that 
$ {\cal E}_k^{\ell}(x_k^{i},Z_{k+1})  \sim  {\cal N}\big(m_k^{\ell}(x_k^i); \Delta  \vert  \sigma_k^{\ell {\scriptscriptstyle \bullet}}(x_k^{i}) \vert_{_2}^2 \big).$
\end{proof}

Notice that the companion probability $\mathbb P(\tilde X_{k+1}^{\ell} \in C_{j_{\ell}}(\Gamma_{k+1}))$  is given, for every  $\ell \in \{1,\ldots,d\}$ and for every  $j_{\ell} \in \{1,\cdots,N_{k+1}^{\ell}\}$, by 
\setlength\arraycolsep{0.3pt}
\begin{eqnarray} 
\mathbb P \big(\tilde X_{k+1}^{\ell} \in C_{j_{\ell}}(\Gamma_{k+1}^{\ell}) \big)  & =&  \sum_{i \in I_k}  \Bigg[  \Phi_0 \Bigg(  \frac{x_{k+1}^{\ell,j_{\ell}+1/2} - m_k^{\ell}(x_k^{i})}{ \sqrt{\Delta}  \vert  \sigma_k^{\ell {\scriptscriptstyle \bullet}}(x_k^{i}) \vert_{_2}}  \Bigg) \nonumber \\
&  & \qquad \quad  -  \   \Phi_0 \Bigg(  \frac{x_{k+1}^{\ell,j_{\ell}-1/2} - m_k^{\ell}(x_k^{i})}{ \sqrt{\Delta}  \vert  \sigma_k^{\ell {\scriptscriptstyle \bullet}}(x_k^{i}) \vert_{_2}}  \Bigg)  \Bigg]
\mathbb P \big(\hat X_{k} = x_k^{i} ).  \label{EqProbaCompos_i}
\end{eqnarray}

We may note that the  $\ell$-th  component process $(\hat X^{\ell}_k)_{k \ge 0}$ is not a Markov chain.  We may however compute the transition probabilities 
\[
\mathbb P(\hat X_{k+1}^{\ell}  = x_{k+1}^{\ell,j_{\ell}} \vert \hat X_k^{\ell'} = x_{k}^{\ell',j_{\ell'}}),  \quad \ell,\ell' \in \{1, \ldots,d \},  \ j_{\ell} \in \{1, \ldots,N_{k+1}^{\ell}\}, \  j_{\ell'} \in \{1, \ldots,N_k^{\ell'}\}.
 \]
This is the aim of the following remark which follows from {\it Bayes} formula. 

\begin{rem}
For $\ell,\ell' \in \{1, \ldots,d \}$,  $ j_{\ell} \in \{1, \ldots,N_{k+1}^{\ell}\}$ and $ j_{\ell'} \in \{1, \ldots,N_k^{\ell'}\}$, we have
\begin{equation}
\mathbb P \big(\hat X_{k+1}^{\ell}  = x_{k+1}^{\ell,j_{\ell}} \vert \hat X_k^{\ell'} = x_{k}^{\ell',j_{\ell'}} \big) = \sum_{i \in I_k} \mathds{1}_{\{i_{\ell'} = j_{\ell'} \}}  \frac{\mathbb P(\hat X_{k+1}^{\ell} = x_{k+1}^{\ell,j_{\ell}} \vert \hat X_k =x_k^{i})}{\mathbb P(\hat X_k^{\ell'} = x_{k}^{\ell',j_{\ell'}})}  \mathbb P(\hat X_k = x_k^{i})
\end{equation}
where the terms $\mathbb P(\hat X_k = x_k^{i})$,  $\mathbb P(\hat X_{k+1}^{\ell} = x_{k+1}^{\ell,j_{\ell}} \vert \hat X_k =x_k^{i})$ and $\mathbb P(\hat X_k^{\ell'} = x_{k}^{\ell',j_{\ell'}})$ are computed from \eqref{EqProbaVector}, \eqref{EqProbaCondCompos_i} and \eqref{EqProbaCompos_i}, respectively.  

As a matter of fact, applying {\it Bayes} formula and summing over $i \in I_k$ yields:
\begin{eqnarray*}
\mathbb P \big(\hat X_{k+1}^{\ell} = x_{k+1}^{\ell,j_{\ell}} \vert \hat X_k^{\ell'} = x_{k}^{\ell',j_{\ell'}} \big) & = &\sum_{i \in I_k}   \frac{\mathbb P(\hat X_{k+1}^{\ell} = x_{k+1}^{\ell,j_{\ell}},  \hat X_k^{\ell'} = x_{k}^{\ell',j_{\ell'}},  \hat X_k =x_k^{i} )}{\mathbb P(\hat X_k^{\ell'} = x_{k}^{\ell',j_{\ell'}})}   \\
& = & \sum_{i \in I_k}  \mathds{1}_{\{i_{\ell'} = j_{\ell'} \}}    \frac{\mathbb P(\hat X_{k+1}^{\ell} = x_{k+1}^{\ell,j_{\ell}},  \hat X_k =x_k^{i} )}{\mathbb P(\hat X_k^{\ell'} = x_{k}^{\ell',j_{\ell'}})} \\
& =& \sum_{i \in I_k} \mathds{1}_{\{i_{\ell'} = j_{\ell'} \}}  \frac{\mathbb P(\hat X_{k+1}^{\ell} = x_{k+1}^{\ell,j_{\ell}} \vert \hat X_k =x_k^{i})}{\mathbb P(\hat X_k^{\ell'} = x_{k}^{\ell',j_{\ell'}})}  \mathbb P(\hat X_k = x_k^{i}).
\end{eqnarray*}
\end{rem}

\medskip 

In the foregoing, we assume that we have access to the  $N_k^{\ell}$-quantizers $\Gamma_k^{\ell}$   of   the $\ell$-th component  $\bar X_k^{\ell}$ of the vector $\bar X_k$,  for every $\ell=1, \ldots,d$. We show how to compute  the distortion functions associated with every component of the vector $\tilde X_{k+1}$, $k=0, \ldots,n-1$. From the numerical point of view,  this will allow us to use the Newton-Raphson algorithm to compute the optimal quantizers associated to each component $\tilde X_{k+1}^{\ell}$, $\ell=1, \ldots,d$,  of  the vector $\tilde X_{k+1}$, for $k=0, \ldots,n-1$. Then,  the quantization $\hat X_{k+1}$ of $\tilde X_{k+1}$  is defined as the product quantization $\hat X_k = (\hat X_k^1, \ldots, \hat X_k^d)$, where  $\hat X_k^{\ell} = \textrm{Proj}_{\Gamma_{k+1}^{\ell}} (\tilde X_{k+1}^{\ell})$.

\subsection{Computing the distortion, the gradient and the Hessian matrix associated to a componentwise quantizer} 
Our aim, for numerical computation of the  componentwise optimal quantizations, is to use   the Newton-Raphson's algorithm  in $\mathbb R^{N_k}$ which involves  the gradient and the Hessian matrix of the distortion  functions $\tilde D_k^{\ell}$, $k=0, \ldots,n$;  $\ell=1,\ldots,d$.   In  the following, we give  useful expressions for the distortion functions   $\tilde D_k^{\ell}$, their gradient vectors $\nabla \tilde D_k^{\ell}$ and  their Hessian matrices $\nabla^2 \tilde D_k^{\ell}$. We state these results in the next  proposition.

Above all, recall that for every $\ell=1, \ldots,d$, for every $k=0, \ldots, n-1$,
\[
  \tilde D_{k+1}^{\ell} (\Gamma_{k+1}^{\ell})  =   \sum_{i  \in I_k}  \mathbb E \big[   d({\cal E}_k^{\ell}(x_k^{i},Z_{k+1}), \Gamma_{k+1}^{\ell})^2 \big] \mathbb P\big(\hat X_k =x_k^{i} \big)
\]
 and notice  that using  Proposition \ref{PropDifferentiability},   the distortion  function  $\tilde D_{k+1}^{\ell}(\Gamma_{k+1}^{\ell}) $   is continuously differentiable as a function of the $N_{k+1}$-quantizer $\Gamma_{k+1}^{\ell} =\{x_{k+1}^{\ell,j_{\ell}}, \ j_{\ell} =1, \ldots N_{k+1}^{\ell}  \}$ (having pairwise distinct components so that it can be viewed as an $N_{k+1}^{\ell}$-tuple) and its gradient  vector reads  
\[
  \nabla  \tilde D_{k+1}^{\ell}( \Gamma_{k+1}^{\ell}) 
   =    2 \Bigg[     \sum_{i \in I_k}  \mathbb E \Big(   \mathds{1}_{\{ {\cal E}_k^{\ell}(x_{k}^{i},Z_{k+1})   \in C_{j_{\ell}}(\Gamma_{k+1}^{\ell})  \}}  \big( x_{k+1}^{\ell,j_{\ell}}   - {\cal E}_k^{\ell}(x_{k}^{i},Z_{k+1})  \big)  \Big) \mathbb P(\hat X_k =x_{k}^{\ell})   \Bigg]_{j_{\ell}=1,\cdots,N_{k+1}^{\ell}}.
\]

We recall that  key point of our method is to deal with the product quantization of the components of the process $(\bar X_k)_{0 \le k \le n}$. From a numerical point of view,  each component will be quantized using the Newton-Raphson algorithm.   To this end, we have to compute  (explicitly) the distortion function $\tilde D_{k+1}^{\ell} (\cdot)$, the components of its gradient vector and the components its Hessian matrix. This is the purpose of the following remark.  Its  proof  relies on tedious though elementary  computation. Therefore,   we have deliberately omitted the proof.  

\begin{rem}  Recall that for every $\ell \in \{ 1, \ldots,d \}$,   $\vartheta^{\ell}_{k}(x)^{2} =   \sum_{p =1}^q \Delta \big(\sigma_k^{\ell p}(x) \big)^2$.

\noindent $a)$ {\em Distortion}.  We have for every $\ell=1, \ldots,d$ and  every $k=0, \ldots, n-1$,
\begin{equation}
\tilde D_{k+1}^{\ell} (\Gamma_{k+1}^{\ell})  =  \sum_{j_{\ell} =1}^{N_{k+1}^{\ell}}   \sum_{i \in I_k}  \Psi_{\ell,j_{\ell}} (x_k^{i}) \,  p_k^i  = \sum_{j_{\ell} =1}^{N_{k+1}^{\ell}}   \mathbb E \Psi_{\ell,j_{\ell}} (\hat X_k),
\end{equation}
where  for every  $x \in \mathbb R^d$,
\setlength\arraycolsep{0.3pt}
\begin{eqnarray*}
\Psi_{\ell,j_{\ell}} (x) &=& \Big(  \big(m_k^{\ell} (x)   - x_{k+1}^{\ell,j_{\ell}}  \big)^2 +  \vartheta^{\ell}_{k}(x)^2  \Big)  \Big(  \Phi_0 \big(x_{k+1}^{\ell,j_{\ell}+}(x) \big)  -   \Phi_0 \big(x_{k+1}^{\ell,j_{\ell}-}(x) \big) \Big) \\
& & \ +  \ 2\vartheta^{\ell}_{k}(x)  \Big( x_{k+1}^{\ell,j_{\ell}}  - m_k^{\ell}(x)   \Big) \Big( \Phi_0' \big( x_{k+1}^{\ell,j_{\ell}+}(x) \big) - \Phi_0' \big(x_{k+1}^{\ell,j_{\ell}-}(x) \big)\Big) \\
& & \ - \  \vartheta^{\ell}_{k}(x) ^2 \Big ( x_{k+1}^{\ell,j_{\ell}+}(x)  \Phi_0' \big( x_{k+1}^{\ell,j_{\ell}+}(x) \big) - x_{k+1}^{\ell,j_{\ell}-}(x)  \Phi_0' \big(x_{k+1}^{\ell,j_{\ell}-}(x) \big) \Big).
\end{eqnarray*}

\noindent $b)$ {\em Gradient}. The components  of  the gradient  $\nabla\tilde D_{k+1}^{\ell}(\Gamma_{k+1}^{\ell})$ are given for every $j_{\ell}=1, \ldots,N_{k+1}^{\ell}$   by
\begin{equation}
  \frac{\partial \tilde D_{k+1}^{\ell}(\Gamma_{k+1}^{\ell}) }{\partial  x_{k+1}^{\ell,j_{\ell}}}  =  \sum_{i \in I_k}   \Psi_{j_{\ell}}' (x_k^{i}) \, p_k^i  =  \mathbb E \Psi_{j_{\ell}}' (\hat X_k)
   \end{equation}
  where for every  $x \in \mathbb R^d$,
  \setlength\arraycolsep{1pt}
\begin{eqnarray*}   
  \Psi_{j_{\ell}}'(x) & = &   \big( x_{k+1}^{\ell,j_{\ell}} - m_k^{\ell}(x) \big)  \Big( \Phi_0 \big(x_{k+1}^{\ell,j_{\ell}+}(x) \big)  -   \Phi_0 \big(x_{k+1}^{\ell,j_{\ell}-}(x) \big)   \Big) \\
  &  &   + \  \vartheta^{\ell}_{k} ( x) \Big(\Phi_0' \big( x_{k+1}^{\ell,j_{\ell}+}(x) \big) - \Phi_0' \big(x_{k+1}^{\ell,j_{\ell}-}(x) \big) \Big).
  \end{eqnarray*}
$c)$ {\em Hessian}. The   sub-diagonal, the super-diagonals and   the  diagonal terms  of the Hessian matrix are given  respectively by  
  \begin{eqnarray*}   
  & &  \frac{\partial^2 \tilde D_{k+1}^{\ell}(\Gamma_{k+1}^{\ell}) }{\partial  x_{k+1}^{\ell,j_{\ell}} \partial  x_{k+1}^{\ell,j_{\ell}-1} } =   \sum_{i \in I_k}    \Psi_{j_{\ell},j_{\ell}-1}'' (x_k^{i}) \, p_k^i  = \mathbb E  \Psi_{j_{\ell},j_{\ell}-1}'' (\hat X_k), \\
  & &  \frac{\partial^2 \tilde D_{k+1}^{\ell}(\Gamma_{k+1}^{\ell}) }{\partial  x_{k+1}^{\ell,j_{\ell}} \partial  x_{k+1}^{\ell,j_{\ell}+1} } =  \sum_{i \in I_k}  \Psi_{j_{\ell},j_{\ell}+1}'' (x_k^{i}) \, p_k^i  = \mathbb E  \Psi_{j_{\ell},j_{\ell}+1}'' (\hat X_k)\\
   & \textrm{ and } \quad &   \frac{\partial^2 \tilde D_{k+1}^{\ell}(\Gamma_{k+1}^{\ell}) }{\partial^2  x_{k+1}^{\ell,j_{\ell}}} =  \sum_{i \in I_k}   \Psi_{j_{\ell},j_{\ell}}'' (x_k^{i}) \, p_k^i  = \mathbb E  \Psi_{j_{\ell},j_{\ell}}'' (\hat X_k)
   \end{eqnarray*}
 where,  for every  $x \in \mathbb R^d$, 
 \begin{eqnarray*}   
 & &  \Psi_{j_{\ell},j_{\ell}-1}'' (x)   = - \frac{1}{4}  \frac{1}{\vartheta^{\ell}_{k}( x)} ( x_{k+1}^{\ell,j_{\ell}} - x_{k+1}^{\ell,j_{\ell}-1}) \Phi_0' \big(x_{k+1}^{\ell,j_{\ell}-}(x) \big), \\
 & & \Psi_{j_{\ell},j_{\ell}+1}'' (x)   = - \frac{1}{4}  \frac{1}{\vartheta^{\ell}_{k}( x)} ( x_{k+1}^{\ell,j_{\ell}+1} - x_{k+1}^{\ell,j_{\ell}}) \Phi_0' \big(x_{k+1}^{\ell,j_{\ell}+}(x) \big),  \\
 & \qquad & \Psi_{j_{\ell},j_{\ell}}'' (x)  = \Phi_0 \big(x_{k+1}^{\ell,j_{\ell}+}(x) \big)  -   \Phi_0 \big(x_{k+1}^{\ell,j_{\ell}-}(x) \big) +  \Psi_{j_{\ell},j_{\ell}-1}'' (x) +  \Psi_{j_{\ell},j_{\ell}+1}'' (x).
   \end{eqnarray*}

\end{rem}

Once we have access to the gradient vector and the Hessian matrix associated with $\tilde X_{k+1}^{\ell}$ and to the optimal grids and companions weights associated with the $\hat X_{k'}$'s, $k'=0, \ldots,k$,   it is possible to write down  (at least formally) a Newton-Raphson zero search procedure to compute the optimal quantizer $\Gamma_{k+1}^{\ell}$. The Newton-Raphson algorithm is in fact   indexed by $ p \ge 0$, where  a current grid $\Gamma_{k+1}^{\ell,p}$ is updated as follows:
\begin{equation}
\Gamma_{k+1}^{\ell,p+1}   =    \Gamma_{k+1}^{\ell,p} - \big(\nabla^2  \tilde D_{k+1}^{\ell} (\Gamma_{k+1}^{\ell,p}) \big)^{-1} \nabla \tilde D_{k+1}^{\ell}( \Gamma_{k+1}^{\ell,p}), \quad  p\ge 1,
\end{equation}
starting  from a $\Gamma^{\ell,0}_{k+1}\!\in \mathbb R^{N_{k+1}^{\ell}}$ (with increasing components).

\medskip 

\begin{rem}  ({\it  Stationarity property})  If  $\Gamma_{k+1}^{\ell}$ is an optimal Markovian product   quantizer for $\tilde X_{k+1}^{\ell}$  and  if $\hat X_{k+1}^{\ell}$ denotes   the  quantization of $\tilde X_{k+1}^{\ell}$ by  the grid $\Gamma_{k+1}^{\ell}$, then   $\Gamma_{k+1}^{\ell}$ is a stationary quantizer for $\tilde X_{k+1}^{\ell}$, means,    $ \mathbb E\Big(\tilde X_{k+1}^{\ell} \big \vert  \hat X_{k+1}^{\ell} \Big)  = \hat X_{k+1}^{\ell}$. Equivalently, this means that if  $\Gamma_{k+1}^{\ell}= \big\{x_{k+1}^{\ell,j_{\ell}}, \, j_{\ell}=1,\ldots,N_{k+1}^{\ell} \big\}$ with 
\begin{eqnarray}  
 x_{k+1}^{\ell,j_{\ell}}   & = & \frac{ \sum_{i \in I_k}  \mathbb E \big( {\cal E}_k^{\ell}(x_{k}^{i},Z_{k+1})  \mathds{1}_{\{ {\cal E}_k^{\ell}(x_{k}^{i},Z_{k+1}) \in C_{j_{\ell}}(\Gamma_{k+1}^{\ell}) \} }\big)   \mathbb P(\hat X_k = x_k^{i})    }{p_{k+1}^{j_{\ell}}} \label{EqLloyd1}
\\
 \textrm{and }  \quad  p_{k+1}^{j_{\ell}}&=&   \sum_{i \in I_k}  \mathbb P \big(   {\cal E}_k^{\ell}(x_{k}^{i},Z_{k+1})  \in C_{j_{\ell}}(\Gamma_{k+1}^{\ell})   \big) \mathbb P(\hat X_k = x_k^{i})   ,\; j_{\ell}=1,\ldots, N_{k+1}^{\ell}.  \label{EqLloyd2}  
 \end{eqnarray}
 
 A straightforward computation leads to the following result: if  $\Gamma_{k+1}^{\ell}= \{x_{k+1}^{\ell,j_{\ell}}, \, j_{\ell}=1,\ldots,N_{k+1}^{\ell}\}$ is a stationary quantizer for $\tilde X_{k+1}^{\ell}$ then, for every $\ell \in \{1, \ldots,d \}$ and for every  $j_{\ell} \in \{  1,\ldots,N_{k+1}^{\ell}\}$,
 \begin{equation}
  x_{k+1}^{\ell,j_{\ell}}   = \frac{ \sum_{i \in I_k} \big[m_k^{\ell}(x_k^i)  \gamma_{\ell,k}(x_k^i)  - \vartheta^{\ell}_{k}(x_k^i)  \gamma_{\ell,k}'(x_k^i)  \big] \, p_k^i }{\sum_{i \in I_k}  \gamma_{\ell,k}(x_k^i)  \, p_k^i  }
 \end{equation}
 where for every $x \in \mathbb R^d$,
   \begin{equation}
 \gamma_{\ell,k}(x) =   \Phi_0 \big(x_{k+1}^{\ell,j_{\ell}+}(x) \big)  -   \Phi_0 \big(x_{k+1}^{\ell,j_{\ell}-}(x) \big) \ \textrm{ and } \  \gamma_{\ell,k}'(x) =   \Phi_0 '\big(x_{k+1}^{\ell,j_{\ell}+}(x) \big)  -   \Phi_0' \big(x_{k+1}^{\ell,j_{\ell}-}(x) \big).
   \end{equation}
\end{rem}

 \subsection{The error analysis}
 
Our aim is now to compute the quadratic quantization error bound  $\Vert  \bar X_T - \hat X_T \Vert_{_2} : = \Vert  \bar X_n - \hat X_n^{\Gamma_n} \Vert_{_2}$.  The analysis of this error bound  is  the subject of the following  theorem.   We suppose that $x_0 = X_0 = \tilde X_0$. We consider here a regular time discretization $(t_k)_{0 \le k \le n}$ with step $\Delta = T/n$: $t_k = \frac{k T}{n}$, $k=0, \ldots,n$.
 
 \begin{thm}  \label{TheoConvergenceRate}  Assume  the coefficients $b$, $\sigma$    satisfy  the  classical Lipschitz assumptions  \eqref{LipAssb}, \eqref{LipAssSig} and  \eqref{EqLinGrowth}.  Let, for every $k=0,\ldots,n$,   $\Gamma_k$ be a Markovian product  quantizer for  $\tilde X_k$ at level  $N_k$. Then,  for  every $k=0, \cdots, n$, for   any $\eta \in ]0,1]$,
 \begin{equation}  \label{EqTheoConvergenceRate}
  \Vert  \bar X_k -  \hat X_k^{\Gamma_k} \Vert_{_2}   \le K_{2, \eta}   \sum_{k'=1}^{k}  e^{ (k-k') \Delta  C_{b,\sigma} }    a_{k'}\big(b, \sigma,d,k,\Delta,x_0 ,L,2+\eta\big)     N_{k'}^{-1/d} 
 \end{equation}
  where for every $ p \in (2,3]$,
  \[
 a_{k'}(b, \sigma, d, k, d, \Delta,x_0 ,L,\theta) :=  e^{\Delta C_{b,\sigma}  \frac{(k-k')}{p} }  \Big[  e^{(\kappa_{p}  + K_{p}) t_{k'}}   \vert  x_{0} \vert^{p}   +   \frac{ d^{k (\frac{p}{2}-1)} e^{\kappa_{p}  \Delta }L + K_{p}}{ d^{\frac{p}{2}-1} (\kappa_{p} +K_{p})}  \big(e^{(\kappa_{p} + K_{p}) t_{k'}} -1 \big) \Big]^{\frac{1}{p}},
  \]
with    $C_{b,\sigma}=[b]_{\rm Lip} + \frac{1}{2} [\sigma]_{\rm Lip}^2$,  $K_{2, \eta}:= K_{2,1,\eta}$ is a universal constant  defined in  Equation  \eqref{LemPierce}; 
 \[
 \kappa_{\theta} := \Big(\frac{(\theta+1)(\theta-2)}{2 } + 2 \theta L \Big)  \ \textrm{ and } \ K_{\theta} := 2^{\theta-1}L{}^{\theta} \Big( 1 + \theta + \frac{p(p-1)}{2} \Delta^{\frac{\theta}{2}-1} \Big)   \mathbb E \vert  Z \vert^{\theta},  ~ Z \sim {\cal N}(0;I_d).
 \]
 \end{thm}
 

\begin{proof}[{\it Proof}. (of Theorem \ref{TheoConvergenceRate})]
Recall that  for every $k \ge 0$, $\hat X_k = (\hat X_k^1, \ldots, \hat X_k^d)$, where  $\hat X_k^{\ell}$ is the quantization of the $\ell$-th component $\bar X_k^{\ell}$ of the vector $\bar X_k$.   Therefore, following step by step the proof of  Lemma 3.2. in \cite{PagSagMQ},  we obtain   for every   $k \ge 1$: 
\[  
\Vert  \bar X_k -  \hat X_k \Vert_{_2} \le  \sum_{k=1}^{k}  e^{ (t_k-t_{k'})   C_{b,\sigma} }  \Vert  \tilde X_{k'} -  \hat X_{k'}^{\Gamma_{k'}} \Vert_{_2},
\]
where $C_{b,\sigma} = [b]_{\rm Lip} + \frac{1}{2} [\sigma]_{\rm Lip}^2$.  Using the definition of $\hat X_k$ combined with Pierce's Lemma (see Theorem \ref{ThmPierceZad}$(b)$) yields
for every $k=1, \ldots, n$,  for any $\eta \in (0,1]$,   
\begin{eqnarray*}
 \Vert  \bar X_k -  \hat X_k \Vert_{_2}  
 & \le  & K_{2,\eta}   \sum_{k'=1}^{k}  e^{ (k-k') \Delta  C_{b,\sigma} }     \Vert \tilde X_{k'}^{\ell}  \Vert_{_{2+\eta}} (N_{k'}^{\ell})^{^{-1/d}}  .
  \end{eqnarray*}
Recall that each component $\hat X_{k}^{\ell}$ of the vector $\hat X_{k}$ is defined as 
  \[
  \widehat X_{k}^{\ell}= \widehat{ \tilde X_k^{\ell}}, \; \ell=1,\ldots,d, 
  \]
where $ \widehat{ \tilde X_k^{\ell}}$ is an optimal quadratic quantization of $\tilde X_k^{\ell}$. Hence each component of $\hat X_{k}$  is stationary with respect to $ \tilde X_k^{\ell}$, that is $ \hat X_k= \mathbb E( \tilde X^{\ell}_{k}\,|\, \widehat{ \tilde X_k^{\ell}})=\widehat{ \tilde X_k^{\ell}}$. We deduce that  
\[
\mathbb E \vert \hat X_{k}\vert^p \le d^{\frac p2-1}\mathbb E \vert \tilde X_{k}\vert^p.
\]

 Following the lines of the proof of Lemma 3.2. in \cite{PagSagMQ}, we  easily show  that for every $\ell \in \{1, \ldots,d\}$, $ \Vert \tilde X_{k'}^{\ell}  \Vert_{_{2+\eta}}  \le a_{k'}(b, \sigma,k, d,\Delta,x_0 ,L,2+\eta)$. 
This completes the proof.
\end{proof}

 \section{Numerical examples} \label{SectNum}

First of all, keep in mind that the computations for all numerical examples  have been performed  on a CPU 2.7 GHz and 4 Go
memory computer.


\subsection{Pricing of a Basket European option}\label{basket}

We consider a European Basket  option  with  maturity $T$ and  strike $K$, based on two stocks with  prices $S^1$ and $S^2$ with associated weights $w_1$ and $w_2$. We suppose that $S^1$ and $S^2$  evolve following the dynamics
\begin{equation}   \label{EqBasketOption}
   \left \{ \begin{array}{l}

 dS_t^1  = r S_t^1  +  \rho \, \sigma_1  S_t^1 dW_{t}^1 +   \sqrt{1-\rho^2}\, \sigma_1 S_t^1  dW_t^2  \\
    dS_t^2  = r S_t^2 dt  + \sigma_2 S_t^2 dW_{t}^1  
    \end{array}  
 \right.
 \end{equation}
where $W^1$ and $W^2$ are two independent Brownian motion,   $r$ is the interest rate and  $\rho \in [-1,1]$, is the correlation term.

  We know that in this case,  the price at time $t=0$ of the call option reads
\begin{equation}   \label{EqpriceBasket}
e^{-rT} \mathbb E \big[ \max( w_1 S^1_T + w_2 S^2_T -K,0) \big] = e^{-rT}  \mathbb E F(X_T), \quad  X = (S^1,S^2), 
\end{equation}
 where the function $F$ is defined, for every $x = (s^1,s^2) \in \mathbb R^2$, by  $F(x)  = \max(w_1s^1 + w_2 s^2 -K,0)$. Using the Markovian product quantization, the price of the Basket European option is approximated by 
 \begin{equation}  \label{EqNumapproxBasket}
e^{-rT} \sum_{j \in I_n} F(x_n^j) \mathbb P(\widehat X_n = x_n^j).
\end{equation}
For the numerical exercises  we will use the following set of  parameters:
\[
r=0.04, \quad \sigma_1 = 0.3, \quad \sigma_2=0.4, \quad \rho=0.5, \quad w_1 = w_2 = 0.5, \quad S^1_0 = 100, \quad S^2_0 =100, \quad T=1.
\]
The benchmark price is given by the algorithm developed in \cite{Ju}.  The results are given in Table \ref{tablebask1}  and  Table \ref{tablebask2}. For both tables,  we consider Call prices for the strikes $K \in \{80, 85, 90, 95, 100 \}$ and Put prices for strikes $K$ lying  to $\{100, 105, 110, 115, 120 \}$. We also depict in Table \ref{tablebask1}  the associated relative error  \big(${\rm Err.} =  {\rm Err.}  = \frac{| \textrm{Price} - MQ_{N}^{n}|}{ \textrm{Price}} $ \big) between the benchmark prices (Price) and the prices obtained from the Markovian and product quantization method of size $N=N_1=N_2$ with $n$ discretization steps (denoted by ${\rm MQ}_{N}^n$) and the computation time (in seconds) for the Markovian and product quantization method. 

In Table   \ref{tablebask1},  we set the number of time steps $n=10$ and make  the sizes $N_1$ et $N_2$ of the marginal quantizers varying  whereas, for the results of Table \ref{tablebask2}, we set $N_1=N_2 =30$ and make varying the number $n$ of time steps.

 We verify, as expected,  that increasing the size of the marginal quantizers lead to more precise results (see Theorem \ref{TheoConvergenceRate}). However, this increases the computation time.  On the other hand, it is also clear from  Theorem \ref{TheoConvergenceRate} that  fixing the marginal quantization size,  the number $n$ of the time  steps increases the global quantization error. 
From the numerical results,  the choice $N_{1}=N_{2}=20$ and $n=10$ seems to be a  good compromise.   
\begin{table}[h!]
\begin{center}
\begin{tabular}{|c|c||c|c||c|c||c|c|}
\hline
$K$ &Price &   ${\rm MQ}_{10}^{10}$  & Err. $(\%)$ & ${\rm MQ}_{20}^{10}$  &Err. $(\%)$ & ${\rm MQ}_{30}^{10}$  & Err. $(\%)$ \\
\hline
$     80 $&$ 25.9491 $&$ 25.4427 $&$ 1.9516 $&$ 25.8721 $&$ 0.2966 $&$ 25.9656 $&$ 0.0636 $\\
$     85 $&$ 22.4481 $&$ 21.9007 $&$ 2.4384 $&$ 22.3543 $&$ 0.4177 $&$ 22.4532 $&$ 0.0229 $\\
$     90 $&$ 19.2736 $&$ 18.6934 $&$ 3.0101 $&$ 19.1596 $&$ 0.5915 $&$ 19.2612 $&$ 0.0645 $\\
$     95 $&$ 16.4323 $&$ 15.8139 $&$ 3.7633 $&$ 16.2935 $&$ 0.8450 $&$ 16.3964 $&$ 0.2183 $\\
$    100 $&$ 13.9197 $&$ 13.2858 $&$ 4.5541 $&$ 13.7537 $&$ 1.1929 $&$ 13.8566 $&$ 0.4535 $\\
\hline
$    100 $&$ 9.9987 $&$ 9.3727 $&$ 6.2602 $&$ 9.8406 $&$ 1.5810 $&$ 9.9435 $&$ 0.5515 $\\
$    105 $&$ 12.6050 $&$ 11.9509 $&$ 5.1895 $&$ 12.4218 $&$ 1.4536 $&$ 12.5218 $&$ 0.6603 $\\
$    110 $&$ 15.5060 $&$ 14.8264 $&$ 4.3828 $&$ 15.2981 $&$ 1.3408 $&$ 15.3965 $&$ 0.7062 $\\
$    115 $&$ 18.6768 $&$ 18.0111 $&$ 3.5646 $&$ 18.4441 $&$ 1.2461 $&$ 18.5422 $&$ 0.7204 $\\
$    120 $&$ 22.0904 $&$ 21.4207 $&$ 3.0317 $&$ 21.8432 $&$ 1.1189 $&$ 21.9345 $&$ 0.7055 $\\
\hline
Time $(s)$& & $0.49$ & & $8.41$ & & $41.82$ & \\
\hline
\end{tabular}
\end{center}
\caption{\footnotesize Prices of a Basket option. The prices correspond to Call prices for strikes $K \in \{80, 85, 90, 95, 100\}$ and to Put prices for strikes $K \in \{ 100, 105, 110, 115, 120 \}$. The  number of time steps $n= 10$. Size of the grids $N_{1}= N_{2}=N$. The error ${\rm Err.}  = \frac{\vert \textrm{Price} - {\rm MQ}_{N}^{n}  \vert}{\textrm{Price}} $ corresponds to the relative error between the benchmark price (Price) and Product Markovian quantization price ${\rm MQ}_{N}^{n}$ of size $N$ with $n$ discretization steps.}\label{tablebask1}
\end{table} 

\begin{table}[h!]
\begin{center}
\begin{tabular}{|c|c||c|c||c|c||c|c|}
\hline
$K$ &Price &   ${\rm MQ}_{30}^{20}$  &Err. $(\%)$ & ${\rm MQ}_{30}^{30}$  &Err. $(\%)$ & ${\rm MQ}_{30}^{40}$  &Err. $(\%)$ \\
\hline
$     80 $&$ 25.9491 $&$ 25.8461 $&$ 0.3969 $&$ 25.7646 $&$ 0.7112 $&$ 25.6937 $&$ 0.9844 $\\
$     85 $&$ 22.4481 $&$ 22.3335 $&$ 0.5107 $&$ 22.2473 $&$ 0.8943 $&$ 22.1731 $&$ 1.2252 $\\
$     90 $&$ 19.2736 $&$ 19.1460 $&$ 0.6618 $&$ 19.0599 $&$ 1.1086 $&$ 18.9837 $&$ 1.5041 $\\
$     95 $&$ 16.4323 $&$ 16.2916 $&$ 0.8561 $&$ 16.2074 $&$ 1.3686 $&$ 16.1310 $&$ 1.8336 $\\
$    100 $&$ 13.9197 $&$ 13.7660 $&$ 1.1041 $&$ 13.6861 $&$ 1.6781 $&$ 13.6117 $&$ 2.2130 $\\
\hline
$    100 $&$ 9.9987 $&$ 9.8490 $&$ 1.4971 $&$ 9.7677 $&$ 2.3095 $&$ 9.6926 $&$ 3.0608 $\\
$    105 $&$ 12.6050 $&$ 12.4428 $&$ 1.2869 $&$ 12.3675 $&$ 1.8842 $&$ 12.2954 $&$ 2.4566 $\\
$    110 $&$ 15.5060 $&$ 15.3328 $&$ 1.1173 $&$ 15.2642 $&$ 1.5593 $&$ 15.1966 $&$ 1.9956 $\\
$    115 $&$ 18.6768 $&$ 18.4954 $&$ 0.9714 $&$ 18.4342 $&$ 1.2990 $&$ 18.3717 $&$ 1.6337 $\\
$    120 $&$ 22.0904 $&$ 21.9045 $&$ 0.8413 $&$ 21.8506 $&$ 1.0854 $&$ 21.7931 $&$ 1.3459 $\\
\hline
Time $(s)$& & $84.13$ & & $151.90$ & & $194.05$ & \\
\hline
\end{tabular}
\end{center}
\caption{\footnotesize Prices of a Basket option. The prices correspond to Call prices for strikes $K \in \{80, 85, 90, 95, 100\}$ and to Put prices for strikes $K \in \{ 100, 105, 110, 115, 120 \}$. Size of the grids $N_{1}= N_{2}=30$ and the  number of time steps $n \in {20, 30, 40}$. The error ${\rm Err.}  =  \frac{\vert \textrm{Price} - {\rm MQ}_{N}^{n} \vert}{\textrm{Price}}$ corresponds to the relative error between the benchmark price (Price) and Product Markovian quantization price ${\rm MQ}_{N}^{n}$ of size $N$ with $n$ discretization steps.}\label{tablebask2}
\end{table} 
   
 
 \subsection{Pricing of a European  option in the Heston model} 

  In this  example, we consider  a European option with maturity $T$ and strike $K$, in the Heston model, introduced in \cite{Heston93}, where  the stock price $S$  and its  stochastic variance $V$ evolve following the dynamics
\begin{equation}  \label{EqHestonIntro}
   \left \{ \begin{array}{l}
   dS_t  = r  S_t dt +   \sqrt{V_t} S_t dW_{t}^1 \\
 dV_t  = \kappa(\theta -V_t) dt  +  \rho \, \sigma  \sqrt{V_t} dW_{t}^1 +   \sqrt{1-\rho^2}\, \sigma \sqrt{V_t}  dW_t^2.  
    \end{array}  
 \right.
 \end{equation}
In the previous equation, the parameter $r$ is the interest rate; $\kappa>0$, is the rate at which  $V$ reverts to the long  running average  variance $\theta>0$;  the parameter $\sigma>0$, is the volatility of the variance and $\rho \in [-1,1]$, is  the correlation term.  In this case, the price of the call  at time $t=0$ reads
\begin{equation} \label{EqpriceHeston}
e^{-rT} \mathbb E \big[\max(S_T - K,0) \big] = e^{-rT} \mathbb E H(X_T), \quad X=(S,V),
\end{equation}
where $H(x) = \max(x^1-K,0)$, for $x = (x^1,x^2) \in \mathbb R^2$. 

 Using the Markovian and product quantization method, the price  of  a call in the Heston model is approximated as
\begin{equation}  \label{EqApproxHestonCall2}
 e^{-rT} \sum_{j_1 = 1}^{N_n^1} \max(s_n^{1j_1}-K,0)  \, \mathbb P(\hat S_n^1 = s_n^{1j_1}),
 \end{equation}
where $\mathbb P(\hat S_n^1 = s_n^{1j_1})$ is computed according to \eqref{EqProbaVector}.\\

For the numerical experiments  we will use the following parameters, obtained from a calibration on market prices:
\[
r=0.04, \quad \kappa = 2.3924, \quad \theta=0.0929,\quad \sigma=0.6903,\quad \rho=-0.82,  \quad S_0 = 100, \quad V_0 =0.0719, \quad T=1.
\]

The pricing of European options under local and stochastic volatility models using recursive quantization techniques has already been studied, see e.g. \cite{PagSagMQ} and \cite{CalFioGra1} for the local volatility case, and \cite{CalFioGra2} for the stochastic volatility case. However, the method we present here is more general and is model free compared to \cite{CalFioGra2} where the method depends on the structure of the model.
   
  The benchmark price is obtained using a Fourier based approach like in \cite{CarrMadan} , since the Heston model is affine and the characteristic function is known in closed form.
   
   Due to the fact that the derivative only depends  on the price and not on the variance, it seems reasonable to choose the marginal grid size $N_1$ greater than $N_2$.  
 To guarantee a good balance between precision and computational time, setting $N_1 = 2 N_2$ seems to be a good trade-off. 
 As for the previous example, we consider Call prices for the strikes $K \in \{80, 85, 90, 95, 100 \}$ and Put prices for strikes $K$ belonging  to $\{100, 105, 110, 115, 120 \}$ and depict in Table   \ref{tableHest1} and Table   \ref{tableHest2}  the corresponding prices (by making varying even $n$ or the sizes $N_1$ and $N_2$ of the quantizers) and the associated relative errors  \big(${\rm Err.}  = \frac{\vert  \textrm{Price} - MQ_{N_1,N_2}^{n}\vert}{\textrm{Price}} $ \big) between the Fourier prices (Price) and the prices obtained from the Markovian and product quantization method of size $N=N_1 \times N_2$ with $n$ discretization steps (denoted by ${\rm MQ}_{N_1,N_2}^n$) and the computation time (in seconds) for the Markovian and product quantization method. 

      
  The best choice from the point of view of accuracy and computational effort, is obtained by taking the size $N_{1}$ of the quantization of the price process $S$ equal to $20$, by taking $N_{2}=10$ for the variance process $V$, and by setting $n=20$.\\

\begin{table}[h!]
\begin{center}
\begin{tabular}{|c|c||c|c||c|c||c|c|}
\hline
$K$ &Price &   ${\rm MQ}_{10,6}^{20}$  & Err. $(\%)$ & ${\rm MQ}_{20,10}^{20}$  &Err. $(\%)$ & ${\rm MQ}_{30,16}^{20}$ &Err. $(\%)$ \\
\hline
$    80 $ &$ 26.3910 $&$ 25.8790 $&$ 1.9401 $&$ 26.2684 $&$ 0.4645 $&$ 26.3705 $&$ 0.0777 $\\
$    85 $ &$ 22.6069 $&$ 22.0686 $&$ 2.3813 $&$ 22.4879 $&$ 0.5264 $&$ 22.5804 $&$ 0.1174 $\\
$    90 $ &$ 19.0506 $&$ 18.5191 $&$ 2.7897 $&$ 18.9360 $&$ 0.6012 $&$ 19.0478 $&$ 0.0148 $\\
$    95 $ &$ 15.7524 $&$ 15.2279 $&$ 3.3295 $&$ 15.6471 $&$ 0.6681 $&$ 15.7634 $&$ 0.0700 $\\
$   100 $ &$ 12.7422 $&$ 12.2750 $&$ 3.6668 $&$ 12.6532 $&$ 0.6987 $&$ 12.7560 $&$ 0.1083 $\\
\hline
$   100 $ &$ 8.8212 $&$ 8.3579 $&$ 5.2515 $&$ 8.7361 $&$ 0.9639 $&$ 8.8390 $&$ 0.2017 $\\
$   105 $ &$ 10.9308 $&$ 10.4676 $&$ 4.2371 $&$ 10.8690 $&$ 0.5655 $&$ 10.9861 $&$ 0.5056 $\\
$   110 $ &$ 13.3794 $&$ 13.0481 $&$ 2.4763 $&$ 13.3460 $&$ 0.2497 $&$ 13.4429 $&$ 0.4741 $\\
$   115 $ &$ 16.1828 $&$ 15.8037 $&$ 2.3425$&$ 16.1796 $&$ 0.0201 $&$ 16.2862 $&$ 0.6389 $\\
$   120 $ &$ 19.3456 $&$ 19.1549 $&$ 0.9855 $&$ 19.3697 $&$ 0.1250 $&$ 19.4603 $&$ 0.5929 $\\
\hline
Time $(s)$& & $0.52$ & & $4.24$ & & $24.82$ & \\
\hline
\end{tabular}
\end{center}
\caption{ \footnotesize Prices in the Heston model. The prices correspond to Call prices for strikes $K \in \{80, 85, 90, 95, 100\}$ and to Put prices for strikes $K \in \{ 100, 105, 110, 115, 120 \}$. The  number of time steps $n= 20$.  The error ${\rm Err.}  = \frac{| \textrm{Price} - {\rm MQ}_{N_1,N_2}^{n}|}{  \textrm{ Price}}$ corresponds to the relative error between the benchmark price (Price) and Product Markovian quantization price ${\rm MQ}_{N_1,N_2}^{n}$ of sizes $N = N_1 \times N_2$ with $n$ discretization steps.}\label{tableHest1}
\end{table}

\begin{table}[h!]
\begin{center}
\begin{tabular}{|c|c||c|c||c|c||c|c|}
\hline
Strike &Price &  ${\rm MQ}_{20,10}^{10}$  & Err. $(\%)$ & ${\rm MQ}_{20,10}^{30}$  &Err. $(\%)$ & ${\rm MQ}_{20,10}^{40}$  &Err. $(\%)$ \\
\hline
$    80 $ &$ 26.3910 $&$ 26.4632 $&$ 0.2737 $&$ 26.1709 $&$ 0.8339 $&$ 26.0900 $&$ 1.1406 $\\
$    85 $ &$ 22.6069 $&$ 22.7001 $&$ 0.4121 $&$ 22.3799 $&$ 1.0041 $&$ 22.2960 $&$ 1.3752 $\\
$    90 $ &$ 19.0506 $&$ 19.1619 $&$ 0.5841 $&$ 18.8229 $&$ 1.1954 $&$ 18.7418 $&$ 1.6207 $\\
$    95 $ &$ 15.7524 $&$ 15.8851 $&$ 0.8428 $&$ 15.5342 $&$ 1.3851 $&$ 15.4616 $&$ 1.8460 $\\
$   100 $ &$ 12.7422 $&$ 12.9032 $&$ 1.2635 $&$ 12.5458 $&$ 1.5415 $&$ 12.4863 $&$ 2.0083 $\\
\hline
$   100 $ &$ 8.8212 $&$ 8.9902 $&$ 1.9155 $&$ 8.6274 $&$ 2.1964 $&$ 8.5673 $&$ 2.8784 $\\
$   105 $ &$ 10.9308 $&$ 11.1520 $&$ 2.0235 $&$ 10.7704 $&$ 1.4670 $&$ 10.7266 $&$ 1.8678 $\\
$   110 $ &$ 13.3794 $&$ 13.6555 $&$ 2.0633 $&$ 13.2609 $&$ 0.8860 $&$ 13.2347 $&$ 1.0820 $\\
$   115 $ &$ 16.1828 $&$ 16.5099 $&$ 2.0212 $&$ 16.1090 $&$ 0.4562 $&$ 16.0992 $&$ 0.5165 $\\
$   120 $ &$ 19.3456 $&$ 19.7143 $&$ 1.9063 $&$ 19.3123 $&$ 0.1717 $&$ 19.3157 $&$ 0.1545 $\\
\hline
Time $(s)$& & $2.18$ & & $6.69$ & & $8.98$ & \\
\hline
\end{tabular}
\end{center}
\caption{ \footnotesize Prices in the Heston model. The prices correspond to Call prices for strikes $K \in \{80, 85, 90, 95, 100\}$ and to Put prices for strikes $K \in \{ 100, 105, 110, 115, 120 \}$.  The error ${\rm Err.}  = \frac{| \textrm{Price} - {\rm MQ}_{N_1,N_2}^{n}|}{  \textrm{ Price}}$ corresponds to the relative error between the benchmark price (Price) and Product Markovian quantization price ${\rm MQ}_{N_1,N_2}^{n}$ of sizes $N_1=20,  N_2=20$ with $n$ discretization steps.}\label{tableHest2}
\end{table}

\subsection{Approximation of BSDE}

 In this section, we  consider a  Markovian BSDE 
   \begin{equation}   \label{EqBSDEIntro1}
 Y_t  =  \xi  + \int_t^T f(s,X_s,Y_s,Z_s) ds - \int_t^T Z_s \cdot dW_s,\quad t \in [0,T], 
 \end{equation}
 where  $W$ is a $q$-dimensional Brownian motion,  $(Z_t)_{t \in [0,T]}$ is a square integrable progressively measurable process taking values in $\mathbb R^q$,   $f : [0,T]{\small \times} \mathbb R^d $  ${\small \times} \mathbb R {\small \times \mathbb R^q} \rightarrow \mathbb R$.   The terminal condition is of the form  $\xi = h(X_T)$, for a given Borel function $h:\mathbb R^d \rightarrow \mathbb R$,  where  $X_T$ is the value at time $T$ of a Brownian diffusion process $(X_t)_{t \geq 0}$,  strong solution to the  stochastic differential equation:
 \begin{equation} \label{EqXIntro}
 X_t  = x  +  \int_0^t b(s,X_s)ds  + \int_0^t \sigma(s,X_s)  dW_s,   \qquad x \in \mathbb R^d.
 \end{equation}   
As pointed out in the introduction, many (time) discretization schemes  and several (spacial) numerical approximation method of the solution of such as BSDE are proposed in the literature (we refer for example to~\cite{BalPagPri0, BouTou, CriManTou, GobTur, HuNuaSon, GobLopTurVaz, BenDen, PagSagBSDE}).  Our aim in this section is to test the performances of our method to the numerical scheme proposed in ~\cite{PagSagBSDE}.  To this end, we first show that  the Markovian product quantization method allows us to compute the term appearing in the  numerical schemes proposed in~\cite{PagSagBSDE} (as well as for several numerical schemes)  using (semi)-closed formula. We then test the performance of our method to a BSDE associated to the price of the Call option in the Black-Scholes model and to a multidimensional  BSDE.

\subsubsection{Explicit numerical scheme for the BSDE}
 Let us  set  for $i \in I_k$, $j \in I_{k+1}$,
\begin{eqnarray*}
& & p_k^i  = \mathbb P(\hat X_k  = x_k^i), \  k=0, \cdots,n \\
& \textrm{ and } \quad  &  p_k^{ij} = \mathbb P(\hat X_{k+1} = x_{k+1}^j \vert \hat X_{k} = x_k^i),  \ k=0, \cdots, n-1.
\end{eqnarray*}
Setting $\hat Y_k = \hat y_k(\hat X_k)$, for every $k \in \{ 0, \cdots,n\}$, the quantized BSDE scheme reads as
\begin{equation*}
   \left \{ \begin{array}{l}
   \hat y_{n}(x_n^i)  =     h( x_{n}^i) \hspace{6.04cm }  x_n^i  \in\Gamma_n\\
   \hat y_{k}(x_k^i)   = \hat {\alpha}_k(x_k^i)   + \Delta_n f \big(t_k, x_{k}^i, \hat {\alpha}_k(x_k^i) , \hat {\beta}_k(x_k^i)  \big) \qquad  \   x_k^i \in \Gamma_k
    \end{array}  
 \right.
 \end{equation*}
 where for $\ k=0, \ldots,n-1$,
\begin{equation}
\hat{\alpha}_k(x_k^i) = \sum_{j \in I_{k+1}}  \hat y_{k+1}(x_{k+1}^j) \,  p_k^{ij} \quad \textrm{ and } \quad \hat {\beta}_k(x_k^i)   =  \frac{1}{\sqrt{\Delta_n}} \sum_{j\in I_{k+1}}    \hat y_{k+1}(x_{k+1}^j) \, \Lambda_k^{ij},
\end{equation}
with 
\[
\Lambda_{k}^{ij}    =  \mathbb{E}\big(Z_{k+1}\mathds { 1}_{\{ \hat X_{k+1}=x_{k+1}^j  \}}\big \vert  \hat X_k=x_k^i \big).
\]

In the following, we give closed  formula for the $\Lambda_{k}^{ij}$'s. We will first suppose  that the components of $\bar X_k$ are independent, for every $k=0, \ldots,n$.

\begin{prop} Suppose that $q=d$ and  ${\cal E}_k^{\ell}(x, Z_{k+1}) = {\cal E}_k^{\ell}(x, Z_{k+1}^{\ell})$, for every $\ell \in \{1, \ldots, d\}$ and $x \in \mathbb R^d$. Then 
\begin{equation}
\Lambda^{ij, \ell}_k = \big(\Phi'_0(x_{k+1}^{\ell,j_{\ell}-}(x_k^i)) -   \Phi'_0(x_{k+1}^{\ell,j_{\ell}+}(x_k^i)) \big)   \prod_{\ell' \not=\ell}^d  \big[\Phi_0 \big( x_{k+1}^{\ell',j_{\ell'}+}(x_k^{i}) \big)  -   \Phi_0 \big( x_{k+1}^{\ell',j_{\ell'}-}(x_k^{i}) \big) \big]
\end{equation}
\end{prop}

\begin{proof}  [{\it \textbf{Proof.}}] Let us set $v^{\ell,j_{\ell}+}:=x_{k+1}^{\ell,j_{\ell}+1/2}$ and $v^{\ell,j_{\ell}-} = x_{k+1}^{\ell,j_{\ell}-1/2}$. 
We have 
\setlength\arraycolsep{1pt}
\begin{eqnarray*}
\Lambda^{ij}_k &=& \mathbb E\big(Z_{k+1} \mathds{1}_{\{ {\cal E}_k(\hat X_k,Z_{k+1})  \in C_j(\Gamma_{k+1})  \}} \big\vert \hat X_k = x_k^i \big) \\
& = & \mathbb E\big(Z_{k+1} \mathds{1}_{\{ {\cal E}_k(x_k^i,Z_{k+1})  \in C_j(\Gamma_{k+1})  \}} \big) \\
& = & \mathbb E\Big(Z_{k+1} \prod_{\ell'=1}^d \mathds{1}_{ \big\{ {\cal E}_k^{\ell'}(x_k^i,Z_{k+1}^{\ell'})   \in \big(v^{\ell',j_{\ell'}-}, v^{\ell,j_{\ell'}+} \big)   \big\}} \Big).
\end{eqnarray*}
Since the components of $Z_{k+1}$ are independent,  it follows that for $\ell=1, \ldots, d$,  the component $(i_{\ell}, j)$ of $\Lambda_k^{ij}$ reads
\begin{eqnarray*}
\Lambda^{ij,\ell }_k &=& \mathbb E\Big(Z_{k+1}^{\ell}  \prod_{\ell'=1}^d \mathds{1}_{ \big\{ {\cal E}_k^{\ell'}(x_k^i,Z_{k+1}^{\ell'})   \in \big(v^{\ell',j_{\ell'}-}, v^{\ell,j_{\ell'}+} \big)   \big\}} \Big) \\
& = &   \mathbb E\Big(Z_{k+1}^{\ell}  \mathds{1}_{ \big\{ {\cal E}_k^{\ell}(x_k^i,Z_{k+1}^{\ell})   \in \big(v^{\ell',j_{\ell}-}, v^{\ell,j_{\ell}+} \big)   \big\}} \Big) \times \mathbb E\Big(\prod_{\ell' \not= \ell}^d \mathds{1}_{ \big\{ {\cal E}_k^{\ell'}(x_k^i,Z_{k+1}^{\ell'})   \in \big(v^{\ell',j_{\ell'}-}, v^{\ell',j_{\ell'}+} \big)   \big\}} \Big).
\end{eqnarray*}
It is clear that 
\[
\mathbb E\Big(\prod_{\ell' \not= \ell}^d \mathds{1}_{ \big\{ {\cal E}_k^{\ell'}(x_k^i,Z_{k+1}^{\ell'})   \in \big(v^{\ell',j_{\ell'}-}, v^{\ell',j_{\ell'}+} \big)   \big\}} \Big) =  \prod_{\ell' \not=\ell}^d  \big[\Phi_0 \big( x_{k+1}^{\ell',j_{\ell'}+}(x_k^{i},0) \big)  -   \Phi_0 \big( x_{k+1}^{\ell',j_{\ell'}-}(x_k^{i},0) \big) \big]. 
\]
On this other hand,
\[
  \mathbb E\Big(Z_{k+1}^{\ell}  \mathds{1}_{ \big\{ {\cal E}_k^{\ell}(x_k^i,Z_{k+1}^{\ell})   \in \big(v^{\ell,j_{\ell}-}, v^{\ell,j_{\ell}+} \big)   \big\}} \Big) = \Phi'_0(x_{k+1}^{\ell,j_{\ell}-}(x_k^i,0)) -   \Phi'_0(x_{k+1}^{\ell,j_{\ell}+}(x_k^i,0)). 
  \]
Combining both previous equalities gives the announced result.
\end{proof}

In the following, we compute the $p$-th component $\Lambda_k^{ij,p}$  of $\Lambda_k^{ij}$  in a general setting. Let us set
\begin{eqnarray*}
\mathbb{J}^{0,p}_{k,j_{\ell}}(x) &=& \Big\{ z \in \mathbb R, \quad \sqrt{\Delta} \sigma_k^{\ell p}(x)  z   \in \big(  x_{k+1}^{\ell,j_{\ell}-1/2} - m_{k}^{\ell}(x), x_{k+1}^{\ell,j_{\ell}+1/2} - m_{k}^{\ell}(x)  \big) \Big \} 
\end{eqnarray*}
and
\[
 \mathbb{L}^{0, p}_{k}(x) = \Big\{ \ell \in \{1, \ldots, d \}, \quad  \sum_{p' \not= p} \big(\sigma_k^{\ell p'} (x) \big)^2 = 0 \Big \}. 
 \]
We also set 
\[
 x_{k+1}^{\ell,p, j_{\ell}-}(x,z) = \frac{x_{k+1}^{\ell,j_{\ell}-1/2} - m_{k}^{\ell}(x) - \sqrt{\Delta} \sigma_k^{\ell p}(x) z  }{\sqrt{\Delta}  \Big( \sum_{p' \not=p} \big(\sigma_k^{\ell p'}(x) \big)^2   \Big)^{1/2} };  \  x_{k+1}^{\ell,p, j_{\ell}+}(x,z)= \frac{x_{k+1}^{\ell,j_{\ell}+1/2} - m_{k}^{\ell}(x) - \sqrt{\Delta} \sigma_k^{\ell p}(x) z  }{\sqrt{\Delta}  \Big( \sum_{p' \not=p} \big(\sigma_k^{\ell p'}(x) \big)^2   \Big)^{1/2} }.
\]

\begin{prop} For every $p \in \{1, \ldots, q\}$, the $p$-th component $\Lambda_k^{ij,p}$  of $\Lambda_k^{ij}$ reads
\begin{equation}
\Lambda^{i j, p}_k =   \mathbb E  \ \zeta \prod_{ \ell  \in \mathbb L^{0,p}_k(x^{i}_k)} \mathds{1}_{ \{  \zeta \in \mathbb J^{0,p}_{k,j_{\ell}}(x^{i}_k)  \}}     \big( \Phi_0(\alpha_{j}^p (x_k^{i},\zeta))  -   \Phi_0(\beta_{j}^p(x_k^{i},\zeta))\big) ^{+}, \quad \zeta \sim {\cal N}(0; 1)
\end{equation}
 (convention: $\prod_{\ell \in \emptyset} (\cdot) = 1$) where   for every $x \in \mathbb R^d$ and $z \in \mathbb R$,
\begin{equation}
 \alpha_{j}^p(x,z)=  \sup_{\ell \in  \big(\mathbb L^{0,p}_{k} (x)\big)^c} x_{k+1}^{\ell,p,j_{\ell}-} (x,z) \quad  \textrm{ and } \quad \beta_{j}^p(x,z)= \inf_{\ell \in \big(\mathbb L^{0,p}_{k} (x)\big)^c} x_{k+1}^{\ell,p,j_{\ell}+} (x,z).
\end{equation}
In particular, if  $p \in \{1, \ldots, q  \}$ and if  for every $\ell \in \{1, \ldots, d\}$ there exists ${p' \not= p}$ such that $ \sigma_k^{\ell p'} (x) \not=0$, then,
 \begin{equation}
\Lambda^{i j, p}_k =   \mathbb E\,   \zeta   \big( \Phi_0(\alpha_{j}^p (x_k^{i},\zeta))  -  \Phi_0(\beta_{j}^p(x_k^{i},\zeta))\big)^{+} .
\end{equation}

\end{prop}

\begin{proof}  [{\it \textbf{Proof.}}]
Let us set $v^{\ell,j_{\ell}+}:=x_{k+1}^{\ell,j_{\ell}+1/2}$ and $v^{\ell,j_{\ell}-} = x_{k+1}^{\ell,j_{\ell}-1/2}$. 
We have 
\setlength\arraycolsep{1pt}
\begin{eqnarray*}
\Lambda^{ij}_k & =& \mathbb E\big(Z_{k+1} \mathds{1}_{\{ {\cal E}_k(\hat X_k,Z_{k+1})  \in C_j(\Gamma_{k+1})  \}} \big\vert \hat X_k = x_k^i \big) \\
& = & \mathbb E\big(Z_{k+1} \mathds{1}_{\{ {\cal E}_k(x_k^i,Z_{k+1})  \in C_j(\Gamma_{k+1})  \}} \big) \\
& = & \mathbb E\big(Z_{k+1} \mathds{1}_{A^{i}_{j}} \big) 
\end{eqnarray*}
where 
\setlength\arraycolsep{1pt}
\begin{eqnarray*}
A^{i}_{j}  & = &   \bigcap_{\ell=1}^d \{{\cal E}_k^{\ell}(x_k^{i},Z_{k+1}) \in (v^{\ell,j_{\ell}-}, v^{\ell,j_{\ell}+} )  \} \\
& = &  \bigcap_{\ell=1}^d  \big\{  m_k^{\ell}(x_k^{i}) + \sqrt{\Delta}\sigma_k^{\ell p} (x_k^i) Z_{k+1}^{p}  +  \sum_{p' \not=p}  \sqrt{\Delta}\sigma_k^{\ell p' } (x_k^i) Z_{k+1}^{p'}  \in (v^{\ell,j_{\ell}-}, v^{\ell,j_{\ell}+} )  \big\}.
\end{eqnarray*}
 Then, conditioning by $Z_{k+1}^{p}$ shows that the component $\Lambda_k^{ij, p}$ of $\Lambda_k^{ij}$ reads
\begin{equation*}
\Lambda^{ij,p}_k =  \mathbb E \big( \mathbb E\big(Z_{k+1}^p \mathds{1}_{A^{i}_{j_{\ell}}} \big)  \big \vert Z_{k+1}^{p}\big)  = \mathbb E(\Psi(\zeta)), \qquad \zeta \sim {\cal N}(0, 1),
\end{equation*}
where for every $u$,
\begin{eqnarray*}
\Psi(u)  &=& u\,  \mathbb P \Big(  \bigcap_{\ell=1}^d A_{\ell,k}^p(u)  \Big) 
\end{eqnarray*}
with
\[
A_{\ell,k}^p(u) =  \big\{  m_k^{\ell}(x_k^{i}) + \sqrt{\Delta}\sigma_k^{\ell p} (x_k^i) u  +  \sum_{p' \not=p}  \sqrt{\Delta}\sigma_k^{\ell p' } (x_k^i) Z_{k+1}^{p'}  \in (v^{\ell,j_{\ell}-}, v^{\ell,j_{\ell}+} )  \big\}.
\]
Keep in mind that 
\[
 \sum_{p' \not=p}  \sqrt{\Delta}\sigma_k^{\ell p' } (x_k^i) Z_{k+1}^{p'}   \ \stackrel{{\cal L}}{=}  \   \Big( \Delta \sum_{p' \not=p} \big(\sigma_k^{\ell p' } (x_k^i) \big)^2 \Big)^{1/2} Z, \quad   Z  \sim {\cal N}( 0; 1).
\]
Then, we may write 
\begin{eqnarray*}
\Psi(u)  &=& u\, \mathds {1}_{\big \{    \bigcap_{\ell_0 \in  \mathbb L_k^{0,p} (x_k^i)} A_{\ell_0,k}^p(u)   \big \}} \mathbb P \Big(  \bigcap_{\ell_{+} \in \big( \mathbb L_k^{0,p} (x_k^i) \big)^c} A_{\ell_{+},k}^p(u)  \Big) 
\end{eqnarray*}
with 
\[
 A_{\ell_{+},k}^p(u)  =  \big\{ Z  \in   (x_{k+1}^{\ell,p,j_{\ell}-} (x_k^i,u),  x_{k+1}^{\ell,p,j_{\ell}+} (x_k^i,u)) \big\}, \quad   Z  \sim {\cal N}( 0; 1).
\]
It follows that 
\begin{eqnarray*}
\Psi(u)  &=&
 u\, \mathds {1}_{\big \{    \bigcap_{\ell_0 \in  \mathbb L_k^{0,p} (x_k^i)} A_{\ell_0,k}^p(u)   \big \}} \mathbb P \Big(  \bigcap_{\ell_{+} \in  \big( \mathbb L_k^{0,p} (x_k^i) \big)^c} A_{\ell_{+},k}^p(u)  \Big) \\
&= & u\, \prod_{\ell \in  \mathbb L_k^{0,p} (x_k^i)} \mathds {1}_{\big \{   u  \in \mathbb J_{k,j_{\ell}}^{0,p}(x_k^i) \big \}} \mathbb P \big(Z  \in  \alpha_{j}^p (x_k^{i},u), \beta_{j}^p (x_k^{i},u) \big). 
\end{eqnarray*}
The result follows immediately.
\end{proof}

\subsubsection{Pricing a risk neutral  Black-Scholes Call under the historical probability}
Let $(\Omega, {\cal A}, \mathbb P)$ be a probability space.  We consider a call option with maturity $T$ and strike $K$  on a stock price $(X_t)_{t \in [0,T]}$  with dynamics
\[
dX_t = \mu X_t dt + \sigma X_t dW_t.
\] 
  Considering a self financing portfolio $Y_t$ with  $\varphi_t$ assets and bonds with risk free return $r$.  We know that (see~\cite{ElkPenQue})  the portfolio evolves according to the following dynamics:
  \begin{equation}  \label{EqBSDENum}
   Y_t = Y_T   +  \int_t^T f(Y_s,Z_s) ds - \int_{t}^T Z_s dW_s
  \end{equation}
where the  payoff $Y_T = (X_T - K)^{+}$, the  hedging strategy   $Z_t  = \sigma \varphi_t X_t$ and $f(y,z)  = -ry- \frac{\mu -r}{\sigma} z.$  It is clear that the function $f$ is linear with respect to $y$ and $z$ and, it is  Lipschitz continuous with $[f]_{\rm Lip} = \max(r, \frac{\mu-r}{\sigma})$.  We perform the numerical tests from the algorithm we propose  with the following parameters 
\[
X_0 =100, \quad r = 0.1, \quad \mu=0.2, \quad K=100, \quad T=0.5
\]
and make varying the volatility $\sigma$.

\begin{table}[h!]
\begin{center}
\begin{tabular}{|c|c|c|c|c|c|c|c|}
\hline
$\sigma$ &$\hat Y_0$ ($n=20$) & $\hat Y_0$ ($n=40$) &$Y_0$ &$\hat Z_0$ ($n=20$) & $\hat Z_0$ ($n=40$) & $Z_0$ \\
\hline
$0.05$&$04.97$& $05.01$&  $05.00$  &$ 04.67$ & $04.58$& $04.62$\\
$0.07$&$05.23$& $05.26$& $05.27$  &$   06.04 $&$05.95$&05.95  \\
$0.10$ &$05.81$&   $05.84$& $05.85$  &$  07.83 $& $07.72$ & $07.71$\\
$0.30$ &$10.88$& $10.89$& $10.91$  &$ 19.00  $& $18.91$ &$19.01$\\
$0.40$ &$13.56$& $13.56$& $13.58 $   & $24.91 $&$24.82$ &$24.99$ \\ 
$0.50$ &$16.26$& $16.25$& $16.26$  &$    31.07 $& $30.98$& 31.24\\
\hline
\end{tabular}
\end{center}
\caption{Call price in the BS model:  $N_k = 100, \ \forall k=1, \dots, n; n \in \{20, 40\} $. Computational time:   $< 1$ second for $n=20$ and around $1$ second for $n=40$.}\label{table1}
\end{table}

\subsubsection{Multidimensional example}
We consider the following  example due to J.-F. Chassagneux: let $t \in [0,T]$. Set 
\[
e_t = \exp(W_t^1+ \ldots +W_t^d +t)
\]
 where $W$ is a $d$-dimensional Brownian motion. Consider the following BSDE:
\[
dX_t = dW_t, \qquad -dY_t  =  f(t,Y_t,Z_t) dt - Z_t \cdot dW_t,  \qquad Y_T = \frac{e_T}{1 +e_T},
\]
where $f(t,y,z)  = (z_1+ \ldots +z_d) \big(y - \frac{2+d}{2d} \big)$.  The solution of this BSDE is given by
\begin{equation}
Y_t  = \frac{e_t}{1 +e_t}, \qquad Z_t = \frac{e_t}{(1+e_t)^2}.
\end{equation}

For the numerical experiments, we put the (regular) time discretization mesh to  $n=20$, with discretization step $\Delta$. We  use the uniform dispatching grid allocation and define the quantization $(\hat W_{t_k})_{ 0 \le k \le n}$ of the Brownian trajectories  $(W_{t_k})_{ 0 \le k \le n}$  from the following recursive procedure
\begin{equation}
\hat W_{t_{k+1}} = \hat W_{t_k}  + \sqrt{\Delta} \,  \hat {\varepsilon},
\end{equation}
$\hat W_0 = 0$ and  where  $\hat {\varepsilon}$ is the optimal quantization of the $d$-dimensional standard Gaussian random variable. We choose $t=0.5$, $d=2,3$, so that $Y_0 =0.5$ and $Z_0^{i} = 0.25$, for every $i=1, \ldots, d$. 

Using the Markovian product quantization method we get
\begin{enumerate}
 \item for $d=2$, with $N_1=N_2=30$:  $\hat Y_0 = 0.504$, $\hat Z_0^1 = \hat Z_0^2= 0.24$. The computation time is around $4$ seconds. 
 
 \item for $d=3$, with $N_1=N_2=N_3=15$: $\hat Y_0 = 0.547$, 	$\hat Z_0^1 = \hat Z_0^2 = \hat Z_0^1 =  0.22$. The computation time is around $1$ minute. 
 \end{enumerate}

  Remark that the only reason motivating the choice of this example is the fact that the  considered backward  has  an explicit solution. Nevertheless  our method works for a general local volatility diffusion process $X$. We also note that when choosing the same size $N_i$ for all marginal quantizers, the complexity of the algorithm is equal to $N_i^d$.  This prevents us from going beyond the dimension 3 without increasing significantly   the  computation time. One way of reducing the computation timein dimension $d \ge 4$  may  be to use the parallel computing. 

   
 \newpage
  
\bibliographystyle{plain}
\bibliography{NLfilteringbib}

\end{document}